\documentclass[revies,onefignum,onetabnum]{siamart220329}

\usepackage{braket,amsfonts}
\usepackage{array}
\usepackage{multirow}
\usepackage{amsfonts,amssymb} 
\usepackage{epsfig,mathrsfs} 
\usepackage{float}

\usepackage[caption=false]{subfig}
\usepackage{pgfplots}

\newsiamthm{claim}{Claim}
\newsiamremark{remark}{Remark}
\newsiamremark{hypothesis}{Hypothesis}
\crefname{hypothesis}{Hypothesis}{Hypotheses}

\usepackage{algorithmic}

\usepackage{graphicx,epstopdf}

\Crefname{ALC@unique}{Line}{Lines}

\usepackage{amsopn}

\usepackage{xspace}
\usepackage{bold-extra}
\usepackage[most]{tcolorbox}

\newtheorem{example}{Example}[section] 
\colorlet{texcscolor}{blue!50!black}
\colorlet{texemcolor}{red!70!black}
\colorlet{texpreamble}{red!70!black}
\colorlet{codebackground}{black!25!white!25}

\lstdefinestyle{siamlatex}{%
  style=tcblatex,
  texcsstyle=*\color{texcscolor},
  texcsstyle=[2]\color{texemcolor},
  keywordstyle=[2]\color{texemcolor},
  moretexcs={cref,Cref,maketitle,mathcal,text,headers,email,url},
}

\tcbset{%
  colframe=black!75!white!75,
  coltitle=white,
  colback=codebackground, 
  colbacklower=white, 
  fonttitle=\bfseries,
  arc=0pt,outer arc=0pt,
  top=1pt,bottom=1pt,left=1mm,right=1mm,middle=1mm,boxsep=1mm,
  leftrule=0.3mm,rightrule=0.3mm,toprule=0.3mm,bottomrule=0.3mm,
  listing options={style=siamlatex}
}

\DeclareTotalTCBox{\code}{ v O{} }
{ fontupper=\ttfamily\color{black},
  nobeforeafter,
  tcbox raise base,
  colback=codebackground,colframe=white,
  top=0pt,bottom=0pt,left=0mm,right=0mm,
  leftrule=0pt,rightrule=0pt,toprule=0mm,bottomrule=0mm,
  boxsep=0.5mm,
  #2}{#1}

\patchcmd\newpage{\vfil}{}{}{}
\title{Correction of  weighted and shifted seven-step BDF for parabolic equations with nonsmooth data \thanks{ \funding{This work was supported by the National Natural Science Foundation of China under Grant No.12471381 and Science Fund for Distinguished Young Scholars of Gansu Province under Grant No. 23JRRA1020.}}}

\author{
 Minghua Chen\thanks{School of Mathematics and Statistics, Gansu Key Laboratory of Applied Mathematics and Complex Systems,
 Lanzhou University, Lanzhou 730000, P.R. China. 
Email address: chenmh@lzu.edu.cn}
\and  Jiankang Shi\thanks{Department of Mathematics, Northwest Normal University, Lanzhou 730070,  P.R. China. 
Email address: shijk20@lzu.edu.cn}
\and  Fan Yu\thanks{Institute for Math \& AI, Wuhan University, Wuhan 430072, P.R. China.
Email address: yufan24@whu.edu.cn}
\and  Zhi Zhou\thanks{Department of Applied Mathematics, The Hong Kong Polytechnic University, Kowloon, Hong Kong, P.R. China. 
Email address: zhizhou@polyu.edu.hk}}

\headers{Seventh-order  methods with nonsmooth data}{M. Chen, J. Shi, F. Yu and Z. Zhou}

\ifpdf
\hypersetup{ pdftitle={Up to seventh order time-stepping methods} }
\fi

\begin{document}
\maketitle


\begin{abstract}
It is well known that the seven-step backward difference formula (BDF) is unstable for the parabolic equations, since it is not even zero-stable.
However, a linear combination of two non zero-stable schemes, namely the seven-step BDF and its shifted counterpart, can yield $A(\alpha)$ stable. Based on this observation,  the authors [Akrivis, Chen, and Yu, IMA J. Numer. Anal., DOI:10.1093/imanum/drae089] propose the weighted and shifted seven-step BDF methods for the parabolic equations, which  stability regions are larger than the standard BDF.
Nonetheless, this approach is not directly applicable for the parabolic equations with nonsmooth data, which  may suffer from severe order reduction.
This motivates us to design proper correction time-stepping schemes  to restore the desired $k$th-order convergence rate of the $k$-step weighted and shifted BDF  ($k\leq 7$) convolution quadrature for the parabolic problems.
We prove that the desired $k$th-order convergence can be recovered even if the source term is  discontinuous  and the initial value is nonsmooth data.
Numerical experiments illustrate the theoretical results.
\end{abstract}



\section{Introduction}\label{Sec1}
We study the convolution quadrature generated by the $k$-step  weighted and
shifted backward difference formula (WSBDF$k$, $k\leq 7$),
for solving the parabolic equations  with nonsmooth data or  discontinuous source \cite{Hettlich2001Identification,Thomee2006Galerkin}
\begin{equation}\label{PQ}
\begin{aligned}
&\partial_t u - Au(t) = f(t) \quad  {\rm in} \quad \Omega \times (0,T)
\end{aligned}
\end{equation}
with the initial value  $u(0)=v$. The operator $A$  denotes the Laplacian $\Delta$ on a convex polyhedral domain $\Omega \subset \mathbb{R}^d$  with  homogeneous Dirichlet boundary conditions,
and $\mathcal{D}(A) = H^{1}_{0}(\Omega) \cap H^{2}(\Omega)$, where $H^{1}_{0}(\Omega)$ and $H^{2}(\Omega)$ denote the standard Sobolev spaces.

The nonsmooth data or discontinuous source term problems of the model \eqref{PQ} arise in multi-material systems of physical and chemical applications.
When multiple materials are in contact with each other,
the conductivity coefficient may vary and exhibit discontinuities at the interface of contact,
such as the singularly perturbed parabolic problems \cite{Aarthika2020nonuniform,Priyadarshana2024Robust,Rao2018Numerical} and  heat conduction procedures \cite{Gianni1993Some,Hettlich2001Identification,Zhu2004Convergence}, e.g.,
\begin{equation*}\tag{$*$}
\partial_t u - \Delta u(t) = \chi_{_D} \quad  {\rm in} \quad \Omega \times (0,T).
\end{equation*}
Here $\chi_{_D}$ denotes the characteristic function of the domain $D$, namely,
\begin{equation*}
	\chi_{_D}=
	\left\{
	\begin{split}
		& 1, \quad x \in D,\\
		& 0, \quad x \notin D.
	\end{split}
	\right.
\end{equation*}

Numerical methods for the time discretization  of the parabolic problems \eqref{PQ} by convolution quadrature have been extensively investigated in \cite{Li2021highorder,Lubich1993RungeKutta,Sheen2000parallel,Sheen2003parallel,Thomee2006Galerkin}.
 Convolution quadrature  techniques have found applications in various areas, including the Navier-Stokes equation
\cite{Lubich1993RungeKutta},  subdiffusion equation \cite{Deng2018time,Jin2017Correction,Shi2023Highorder,Yan2018analysis},  and nonlinear fractional problems \cite{Wang2020high}.
No seventh order convolution quadrature  generated by BDF$7$ was given there, since it is not zero-stable.
Recently,  based on the energy technique relying  on suitable multipliers \cite{Akrivis2021energy},  a class of simple and efficient WSBDF$7$ methods  for the parabolic equations \eqref{PQ} were proposed  in
\cite{Akrivis2024weighted}.
The proposed WSBDF$k$ methods \cite{Akrivis2024weighted}, including the standard BDF$k$ \cite{Akrivis2021energy},
have been widely applied to various scientific phenomena, such as mean curvature flow \cite{EGK,KL},  gradient flows \cite{HS}, and fractional equations \cite{CYZ}. Nonetheless, the WSBDF$7$ method \cite{Akrivis2024weighted} is not directly  applicable to the parabolic equations with nonsmooth data, which  may suffer from a severe order reduction, as shown in Table \ref{table:2.1}.

 To fill in this gap,
we propose a smoothing method for time-stepping schemes, where nonsmooth data are regularized by using an $m$-fold integral-differential calculus \cite{Shi2023Highorder},
and the equation is discretized by the $k$-step WSBDF convolution quadrature, referred to as the ID$m$-WSBDF$k$ method ($k\leq 7$). However, this method also exhibits first- or second-order convergence, as observed in Table \ref{table:2.2}. This motivates us to design the corrected  ID$m$-WSBDF$k$ method to restore the desired $k$th-order convergence rate in this work.
	
The outline of this paper is as follows. In the next section, we design the numerical schemes, including the WSBDF$k$, ID$m$-WSBDF$k$, and corrected ID$m$-WSBDF$k$. In Section \ref{Sec3}, we establish the discretized solution representation for the corrected ID$m$-WSBDF$k$. A few technical lemmas are provided in Section \ref{Sec4}. Optimal error estimates with $\mathcal{O}(\tau^{\min\{m+1,k\}})$ are proved in Section \ref{Sec5}. To demonstrate the effectiveness of the presented schemes, results of numerical experiments are reported in Section \ref{Sec6}.

\section{Numerical schemes}\label{Sec2}
Let $V(t) = u(t) - u(0) = u(t) -  v$ with $V(0) = 0$.
Then we can rewrite \eqref{PQ} as
\begin{equation}\label{RPQ1}
V'(t) - A V(t)=  Av + f(t), \quad 0<t\leq T,
\end{equation}
which is equivalent to
\begin{equation}\label{RPQ}
V'(t) - A V(t)=\frac{Av}{\Gamma(m+1)}  \partial^{m}_{t}  t^{m}  + \partial^{m}_{t} F(t), \quad 0<t\leq T.
\end{equation}
Here, the $m$-fold integral calculus of $f(t)$ is defined by \cite[p.64]{Podlubny1999Fractional}
\begin{equation}\label{smoothing1}
F(t)=J^{m}f(t)=
\frac{1}{\Gamma(m)}\int_0^t(t-\tau)^{m-1} f(\tau)d\tau=\frac{t^{m-1}}{\Gamma(m)} \ast f(t), \quad 1\leq m \leq  k\leq 7,
\end{equation}
{ and $\ast$ denotes the Laplace convolution: $(f\ast g)(t) = \int_{0}^{t} f(t-s) g(s) ds$.}

Let $N\in \mathbb{N},$ $\tau=T/N$ be the time step, and $t_{n}=n \tau$, $n=0,\dotsc ,N$ be a uniform partition of the interval $[0,T]$.
We recursively define a sequence of approximations $V^{n}$ to the nodal values $ V(t_{n})$ of the exact solution.
The convolution quadrature generated by $k$-step backward difference formula,  shifted backward difference formula (SBDF$k$) and  weighted and shifted backward difference formula, respectively,  approximates to the derivative $\partial_{t}\varphi(t_n)$, $\partial_{t}\varphi(t_{n-1})$
and $\beta\partial_{t}\varphi(t_n)+(1-\beta)\partial_{t}\varphi(t_{n-1})$
by
\begin{equation}\label{add2.3}
\begin{split}
& {_{b} \partial_{\tau}}\varphi^{n}=\frac{1}{\tau}\sum^{n}_{j=0} \overline{b}_{j} \varphi^{n-j}, ~~
 {_{s}\partial_{\tau}}\varphi^{n} = \frac{1}{\tau}\sum^{n}_{j=0} \overline{s}_{j} \varphi^{n-j} ~~ {\rm and} ~~
 {_{w}\partial_{\tau}}\varphi^{n}=\frac{1}{\tau}\sum^{n}_{j=0} w_{j} \varphi^{n-j}.
\end{split}
\end{equation}
Here  $ \overline{b}_{j}$, $\overline{s}_{j}$ and $w_{j}$, respectively, are the coefficients  in the series expansion of the following characteristic polynomials
\begin{equation}\label{CWSBDF}
\begin{split}
{\rho_{b}}(\xi)
& = \frac{1}{\tau}\sum^{\infty}_{j=0} \overline{b}_{j} \xi^{j}  = \frac{1}{\tau} \sum_{j=1}^{k} \frac{1}{j} (1-\xi)^{j}, \\
{\rho_{s}}(\xi)
& = \frac{1}{\tau}\sum^{\infty}_{j=0} \overline{s}_{j} \xi^{j}  = \frac{1}{\tau} \left( \sum_{j=1}^{k} \frac{1}{j} (1-\xi)^{j} -   \sum_{j=2}^{k} \frac{1}{j-1} (1-\xi)^{j} \right), \\
{\rho_{w}}(\xi)
& =\frac{1}{\tau}\sum^{\infty}_{j=0} w_{j} \xi^{j} = \beta {\rho_{b}}(\xi)   +  (1-\beta) {\rho_{s}}(\xi).
\end{split}
\end{equation}
{
It is worthwhile to note that the second term of the second formula in \eqref{CWSBDF} automatically vanishes when $k=1$.
In fact, the coefficients in \eqref{add2.3} and the characteristic polynomials in \eqref{CWSBDF} depend on $k$, i.e., $ \overline{b}_{j} := \overline{b}_{j}^{(k)}$, $\rho_{b}(\xi):=\rho_{b}^{(k)}(\xi)$ and ${_{b} \partial_{\tau}}:={_{b} \partial_{\tau}^{(k)}}$.
For convenience, and where no confusion arises, we omit the superscript $(k)$ throughout the paper.
}

Then the WSBDF$k$ method for \eqref{RPQ1} is designed by \cite{Akrivis2024weighted}
\begin{equation}\label{WSBDF}
{_{w}\partial_{\tau}} V^{n} - \beta A V^{n} - (1-\beta) A V^{n-1}= \beta Av + (1-\beta) Av + \beta f^{n}+(1-\beta) f^{n-1}, \quad \beta\geq 3
\end{equation}
{ with $V^{0}=0$.}
In fact, $\beta>1/2$ for $k=3,4,5$ and $\beta>2.6$ for $k=6,7$, see Theorem 5 in \cite{Li1991linear}.  For convenience, we take $\beta\geq 3$ throughout this paper.
\begin{remark}(Stability regions)\label{remarkad2.1}
The BDF$k$ methods are $A(\overline{\vartheta}_{k})$-stable with $\overline{\vartheta}_{1}=\overline{\vartheta}_{2}=90^\circ$,
$\overline{\vartheta}_{3}\approx 86.03^\circ, \overline{\vartheta}_{4}\approx 73.35^\circ, \overline{\vartheta}_{5}\approx 51.84^\circ,$ and $\overline{\vartheta}_{6} \approx 17.84^\circ$;
see \cite[Section V.2]{Hairer2010solving}.
The WSBDF$k$ methods are $A({\vartheta}_{k})$-stable with  $\overline{\vartheta}_{k}<  {\vartheta}_{k} < \overline{\vartheta}_{k-1}$ for $k=3,4,5,6.$ It can  be shown that $ {\vartheta}_{k} \to \overline{\vartheta}_{k-1}$ as $\beta\to \infty$
for $k=3,\dotsc,7;$ see Remark  2.1 in \cite{Akrivis2024weighted}.
\end{remark}

We first numerically check the convergence order of schemes \eqref{WSBDF} with nonsmooth data  through an example.
\begin{example}\label{EX1}
Let $T=1$ and $\Omega=(-1, 1)$. Consider the initial-boundary value problem with the initial condition $v(x)= \sqrt{1-x^2}$ and the source function $f(x,t)= 0$.
\end{example}
\begin{table}[H]  
{\small \begin{center}
\caption{The discrete $L^2$-norm $||u^{N}-u^{2N}||$ and convergence order of schemes \eqref{WSBDF} with $\beta=3$.}
\begin{tabular}{|c| c c c c c |c|}
\hline
  $k$    &  $N=100$    & $N=200$     & $N=400$     &  $N=800$    &  $N=1600$   &  Rate        \\ \hline
  1      & 9.8131e-03  & 4.8676e-03  & 2.4228e-03  & 1.2085e-03  & 6.0352e-04  & $\approx$ 1.00         \\ \hline
  2      & 4.2607e-03  & 2.0379e-03  & 9.9764e-04  & 4.9369e-04  & 2.4558e-04  & $\approx$ 1.00         \\ \hline
  3      & 4.2870e-03  & 2.0458e-03  & 9.9977e-04  & 4.9424e-04  & 2.4572e-04  & $\approx$ 1.00         \\ \hline
  4      & 4.2883e-03  & 2.0460e-03  & 9.9978e-04  & 4.9424e-04  & 2.4572e-04  & $\approx$ 1.00         \\ \hline
  5      & 4.2882e-03  & 2.0460e-03  & 9.9978e-04  & 4.9424e-04  & 2.4572e-04  & $\approx$ 1.00        \\ \hline
  6      & 4.2882e-03  & 2.0460e-03  & 9.9978e-04  & 4.9424e-04  & 2.4572e-04  & $\approx$ 1.00         \\ \hline
  7      & 4.2875e-03  & 2.0460e-03  & 9.9978e-04  & 4.9424e-04  & 2.4572e-04  & $\approx$ 1.00        \\ \hline
\end{tabular}\label{table:2.1}
\end{center}}
\end{table}

From Table \ref{table:2.1}, we find that the numerical scheme \eqref{WSBDF} with nonsmooth data   just has first-order convergence rate, which  is suffer from a severe order reduction. In fact,
for homogeneous parabolic problem $f(x,t)=0$, there exists  \cite[p.39]{Thomee2006Galerkin}
$$\| \partial_{t}u(t) \|_{L^{2}(\Omega)} \leq C t^{-1} \|u(0) \|_{L^{2}(\Omega)}.$$  This shows that $u$ has an initial layer at $t \rightarrow 0^{+}$ (i.e., unbounded near $t = 0$).
Hence, the high-order convergence rates may not hold for nonsmooth data.

To restore the desired
$k$th-order convergence rate,   the nonsmooth data would be corrected \cite{Jin2017Correction} at the starting $k-1$ steps    or   regularized \cite{Shi2023Highorder} by using an $m$-fold integral-differential (ID$m$) calculus.
Here we  develop  the ID$m$ method  for nonsmooth data,
and the equation \eqref{RPQ} is discretized by the $k$-step WSBDF convolution quadrature, called ID$m$-WSBDF$k$ method in \eqref{IDWSBDFad1}.
Namely,
\begin{equation}\label{IDWSBDFad1}
\begin{split}
& {_{w}\partial_{\tau}} V^{n} - \beta AV^{n} - (1-\beta)  A V^{n-1} \\
& \!=\!  \frac{\beta Av }{\Gamma(m+1)} {_{b} \partial^{m}_{\tau}} t_{n}^{m}   \! +\!  \frac{(1-\beta) Av}{\Gamma(m+1)}   {_{s}\partial}^{m}_{\tau} t_{n+m-1}^{m}
\!+ \!\beta {_{b} \partial^{m}_{\tau}} F^{n} + (1-\beta)  {_{s}\partial}^{m}_{\tau} F^{n+m-1}.
\end{split}
\end{equation}
The  difference operator  ${_{b}\partial^{m}_{\tau}}$ and
the shifted difference operator$ {_{s}\partial^{m}_{\tau}}$
are defined by
\begin{equation*}
\begin{split}
& {_{b}\partial^{m}_{\tau}} \varphi^{n}=\frac{1}{\tau^{m}}\sum^{n}_{j=0} b_{j} \varphi^{n-j}, ~~
 {_{s}\partial^{m}_{\tau}}\varphi^{n} = \frac{1}{\tau^{m}}\sum^{n}_{j=0} s_{j} \varphi^{n-j}
\end{split}
\end{equation*}
with   the characteristic polynomials
\begin{equation}\label{CWSBDFkm}
\begin{split}
{\rho^{m}_{b}} (\xi)
& = \frac{1}{\tau^{m}}\sum^{\infty}_{j=0} b_{j} \xi^{j}  = \frac{1}{\tau^{m}} \left(\sum_{j=1}^{k} \frac{1}{j} (1-\xi)^{j} \right)^{m}, \\
{\rho^{m}_{s}}(\xi)
& = \frac{1}{\tau^{m}}\sum^{\infty}_{j=0} s_{j} \xi^{j}  = \frac{1}{\tau^{m}} \left( \sum_{j=1}^{k} \frac{1}{j} (1-\xi)^{j} -   \sum_{j=2}^{k} \frac{(1-\xi)^{j}}{j-1} \right)^{m}.
\end{split}
\end{equation}
\begin{table}[H]  
{\small \begin{center}
\caption{The discrete $L^2$-norm $||u^{N}-u^{2N}||$ and convergence order of schemes \eqref{IDWSBDFad1} with $k=7$, $\beta=3$.}
\begin{tabular}{|c| c c c c c |c|}
\hline
  $m$    &  $N=100$    & $N=200$     & $N=400$     &  $N=800$    &  $N=1600$   &  Rate        \\ \hline
  1      & 1.1222e-05  & 1.2056e-06  & 3.0138e-07  & 7.5345e-08  & 1.8836e-08  & $\approx$ 2.00         \\ \hline
  2      & 3.9684e-05  & 1.6570e-05  & 8.1290e-06  & 4.0265e-06  & 2.0039e-06  & $\approx$ 1.00         \\ \hline
  3      & 6.9549e-05  & 3.1113e-05  & 1.5267e-05  & 7.5633e-06  & 3.7643e-06  & $\approx$ 1.00         \\ \hline
  4      & 7.8673e-05  & 3.7233e-05  & 1.8273e-05  & 9.0534e-06  & 4.5061e-06  & $\approx$ 1.00         \\ \hline
  5      & 8.6832e-05  & 4.1921e-05  & 2.0576e-05  & 1.0195e-05  & 5.0744e-06  & $\approx$ 1.00        \\ \hline
  6      & 9.6487e-05  & 4.6661e-05  & 2.2905e-05  & 1.1349e-05  & 5.6493e-06  & $\approx$ 1.00        \\ \hline
  7      & 1.0541e-04  & 5.0733e-05  & 2.4907e-05  & 1.2342e-05  & 6.1434e-06  & $\approx$ 1.00        \\ \hline
  \end{tabular}\label{table:2.2}
\end{center}}
\end{table}

It seems that the ID$m$-WSBDF$7$ method is valid for nonsmooth data, since it  restores  the second-order convergence with $m=1$,  see Table \ref{table:2.2}. However, it still drops down to the first-order convergence with $2\leq m \leq 7$, $k=7$.
This motivates us to design the corrected  ID$m$-WSBDF$k$ method  to restore the desired
$k$th-order convergence rate. Namely,
\begin{equation}\label{IDWSBDF}
\begin{split}
& {_{w}\partial_{\tau}} V^{n} - \beta AV^{n} - (1-\beta)  A V^{n-1} \\
& =  \frac{\beta Av }{\Gamma(m+1)} {_{b} \partial^{m}_{\tau}} t_{n}^{m}    +  \frac{(1-\beta) Av}{\Gamma(m+1)} \left(  {_{s}\partial}^{m}_{\tau} t_{n+m-1}^{m} - \frac{1}{\tau^{m}} \sum^{n+m-1}_{j=n} s_{j} t_{n+m-1-j}^{m} \right) \\
& \quad + \beta {_{b} \partial^{m}_{\tau}} F^{n} + (1-\beta) \left(  {_{s}\partial}^{m}_{\tau} F^{n+m-1} - \frac{1}{\tau^{m}} \sum^{n+m-1}_{j=n} s_{j} F^{n+m-1-j} \right),
\end{split}
\end{equation}
which is equivalent to
\begin{equation}\label{IDWSBDF1}
\begin{split}
& \frac{1}{\tau}\sum^{n}_{j=0} w_{j} V^{n-j} - \beta AV^{n} - (1-\beta)  A V^{n-1} \\
& =  \frac{\beta Av}{\Gamma(m+1)}  \frac{1}{\tau^{m}}\sum^{n}_{j=0} b_{j} t_{n-j}^{m}  + \frac{ (1-\beta) Av }{\Gamma(m+1)} \frac{1}{\tau^{m}} \sum^{n-1}_{j=0} s_{j} t_{n+m-1-j}^{m}  \\
& \quad + \beta \frac{1}{\tau^{m}}\sum^{n}_{j=0} b_{j} F^{n-j} + (1-\beta) \frac{1}{\tau^{m}}\sum^{n-1}_{j=0} s_{j} F^{n+m-1-j},\quad \beta\geq3.
\end{split}
\end{equation}
Note that the corrected  ID$m$-WSBDF$k$ method \eqref{IDWSBDF} reduces to the  ID$m$-WSBDF$k$ method \eqref{IDWSBDFad1} when $s_n=0$ for $n\geq km+1$ with $1\leq m \leq  k\leq 7$. In particular, the corrected  ID$m$-WSBDF$k$ method is equivalent to the ID$m$-WSBDF$k$ method when $m=1$, since $t_0=0$ and $F^0=0$, even though the coefficient $s_n$ may not be zero.
{ When comparing \eqref{IDWSBDF} with \eqref{IDWSBDFad1}, we observe that the modification affects only the first few terms.}

\begin{remark}
For the time semidiscrete schemes \eqref{IDWSBDF}, we require $v\in \mathcal{D}(A)$.
However, one can use the schemes \eqref{IDWSBDF} to derive   error estimates when dealing with   nonsmooth data $v\in L^2(\Omega)$; as illustrated in Theorem \ref{Theorem3.3}.
This paper primarily  focuses on the time semidiscrete schemes \eqref{IDWSBDF}, since the spatial discretization is well understood.
In fact, we can choose $v_h=R_hv$ if $v\in \mathcal{D}(A)$ and $v_h=P_hv$ if $v\in L^2(\Omega)$; see \cite{Shi2022Correction,Thomee2006Galerkin,Wang2020Two}.
\end{remark}

{
\begin{remark}
In fact, the analytical framework can be extended to the case where the operator $A$ generates a bounded holomorphic semigroup on a general Banach space, provided that $A$ satisfies the resolvent estimate \eqref{resolvent estimate}.
\end{remark}
}

\section{Solution representation}\label{Sec3}
It is well known that the operator $A$ satisfies the resolvent estimate \cite{Lubich1996Nonsmooth,Thomee2006Galerkin}
\begin{equation}\label{resolvent estimate}
\left\| (z-A)^{-1} \right\| \leq c |z|^{-1} \quad \forall z\in \Sigma_{\theta}
\end{equation}
for all $\theta\in (\pi/2, \pi)$. Here $\Sigma_{\theta}:=\{ z\in \mathbb{C}\backslash \{0\}:|\arg z| < \theta \}$ is a sector of the complex plane $\mathbb{C}$.
Here and below $\left\|\cdot \right\|$ and $\left\|\cdot \right\|_{L^2(\Omega)}$  denote the operator norm \cite[p.\,91]{Thomee2006Galerkin} and usual norm \cite[p.\,2]{Thomee2006Galerkin} in the space $L^2(\Omega)$, respectively.

Applying the Laplace transform in \eqref{RPQ} yields
\begin{equation*}
\widehat{V}(z)= (z - A)^{-1} \left( z^{-1} Av + \widehat{f}(z) \right).
\end{equation*}
By the inverse Laplace transform, we obtain \cite{Thomee2006Galerkin}
\begin{equation}\label{LT}
\begin{split}
V(t)
&=  \frac{1}{2\pi i} \int_{\Gamma_{\theta, \kappa}} e^{zt} (z - A)^{-1}  \left( z^{-1} Av + \widehat{f}(z) \right) dz \\
&=  \frac{1}{2\pi i} \int_{\Gamma_{\theta, \kappa}} e^{zt} (z - A)^{-1}  \left( z^{-1} Av + z^{m} \widehat{F}(z) \right) dz
\end{split}
\end{equation}
with
\begin{equation}\label{Gamma}
\Gamma_{\theta, \kappa}=\{ z\in \mathbb{C}: |z|=\kappa, |\arg z|\leq \theta \} \cup \{ z\in \mathbb{C}: z=re^{\pm i\theta}, r\geq \kappa \} = \Gamma_{\kappa} \cup \Gamma_{\theta}
\end{equation}
and $\theta \in (\pi/2, \pi)$, $\kappa>0$.

Given a sequence $(\kappa^n)_0^\infty$ we denote by
\begin{equation*}
\widetilde{\kappa}(\xi)=\sum_{n=0}^{\infty}\kappa^n \xi^n
\end{equation*}
its generating power series. Let
 \begin{equation}\label{eq3.4}
 \gamma_{l}(\xi)=\sum^{\infty}_{n=0}n^{l} \xi^{n},\quad l \geq 0.
 \end{equation}
Then we have the following result.

\begin{lemma}\label{Lemma2.1}
Let ${\rho^{m}_{b}}$, ${\rho^{m}_{s}}$ and ${\rho_{w}}$ be given in \eqref{CWSBDFkm} and \eqref{CWSBDF}, respectively. Let $\gamma_{m}(\xi)=\sum^{\infty}_{n=0}n^{m} \xi^{n}$ and $F(t)=J^{m}f(t)$ with $1\leq m \leq k\leq 7$ in \eqref{smoothing1}.
Then the discrete solution of \eqref{IDWSBDF} is represented by
\begin{equation*}
V^{n}=\frac{\tau}{2\pi i}\int_{\Gamma^{\tau}_{\theta,\kappa}} e^{zt_n} \left( \frac{{\rho_{w}} (e^{-z\tau})}{\beta + (1-\beta)e^{-z\tau}} - A\right)^{-1} \left(\beta + (1-\beta)e^{-z\tau}\right)^{-1} G(e^{-z\tau}) dz
\end{equation*}
with $\Gamma^{\tau}_{\theta, \kappa}=\{z\in \Gamma_{\theta, \kappa}: |\Im z|\leq \pi / \tau\}$ and
\begin{equation*}
\begin{split}
G(\xi) :=
&     \frac{\beta Av}{\Gamma(m+1)} {\rho^{m}_{b}} (\xi) \gamma_{m}(\xi) \tau^{m}
  +   \frac{(1-\beta)Av}{\Gamma(m+1)} {\rho^{m}_{s}}(\xi) \xi^{1-m} \tau^{m}  \left( \gamma_{m}(\xi) -  \sum^{m-1}_{n=0} n^{m}  \xi^{n} \right) \\
& + \beta {\rho^{m}_{b}}(\xi) \widetilde{F} (\xi) + (1-\beta) {\rho^{m}_{s}}(\xi) \xi^{1-m}\left( \widetilde{F}(\xi) - \sum^{m-1}_{n=0} F^{n}  \xi^{n} \right).
\end{split}
\end{equation*}
\end{lemma}
\begin{proof}
Multiplying \eqref{IDWSBDF} by $\xi^{n}$ and summing over $n$, we obtain
\begin{equation}\label{addeq3.4}
\begin{split}
& \sum^{\infty}_{n=1} {_{w}\partial_{\tau}} V^{n} \xi^{n} - \beta \sum^{\infty}_{n=1} AV^{n} \xi^{n} - (1-\beta)  \sum^{\infty}_{n=1} A V^{n-1} \xi^{n} \\
& =  \frac{\beta Av }{\Gamma(m+1)} \sum^{\infty}_{n=1} {_{b} \partial^{m}_{\tau}} t_{n}^{m} \xi^{n}
 +  \frac{(1-\beta) Av}{\Gamma(m+1)} \frac{1}{\tau^{m}} \sum^{\infty}_{n=1}    \sum^{n-1}_{j=0} s_{j} t_{n+m-1-j}^{m}  \xi^{n} \\
& \quad + \beta \sum^{\infty}_{n=1} {_{b} \partial^{m}_{\tau}} F^{n} \xi^{n} + (1-\beta) \frac{1}{\tau^{m}} \sum^{\infty}_{n=1}  \sum^{n-1}_{j=0} s_{j} F^{n+m-1-j}  \xi^{n}.
\end{split}
\end{equation}
From \eqref{CWSBDF}, and $V^0=0$, there exists
\begin{equation*}
\sum^{\infty}_{n=1} {_{w} \partial_{\tau} V^{n} } \xi^{n} = \sum^{\infty}_{n=1} \frac{1}{\tau}\sum^{n}_{j=0} w_{j} V^{n-j} \xi^{n}
= \frac{1}{\tau} \sum^{\infty}_{j=0} w_{j} \xi^{j} \sum^{\infty}_{n=0} V^{n} \xi^{n} =  {\rho_{w}} (\xi) \widetilde{V} (\xi).
\end{equation*}
Similarly, using  $F^0=F(0)=0$ and $t_0^m=0$, $m\geq 1$, we get
\begin{equation*}
\begin{split}
& \sum^{\infty}_{n=1} AV^{n}\xi^{n}  = A \widetilde{V} (\xi), \quad
\sum^{\infty}_{n=1}   A V^{n-1} \xi^{n}  = \xi A \widetilde{V} (\xi),  \\
& \sum^{\infty}_{n=1} {_{b} \partial^{m}_{\tau}} t_{n}^{m} \xi^{n} =  {\rho^{m}_{b}}(\xi) \gamma_{m}(\xi) \tau^{m}, \quad
\sum^{\infty}_{n=1} {_{b}\partial^{m}_{\tau}} F^{n} \xi^{n} = {\rho^{m}_{b}}(\xi) \widetilde{F} (\xi)
\end{split}
\end{equation*}
with $\gamma_{m}(\xi)=\sum^{\infty}_{n=0}n^{m} \xi^{n}$. On the other hand, one has
\begin{equation*}
\begin{split}
& \frac{1}{\tau^{m}} \sum^{\infty}_{n=1} \sum^{n-1}_{j=0} s_{j}F^{n+m-1-j}  \xi^{n}
 = \frac{1}{\tau^{m}} \sum^{\infty}_{j=0} \sum^{\infty}_{n=j} s_{j}F^{n+m-j}  \xi^{n+1} \\
&  = \frac{1}{\tau^{m}} \sum^{\infty}_{j=0}  s_{j} \xi^{j} \sum^{\infty}_{n=m} F^{n}  \xi^{n+1-m}
  = {\rho^{m}_{s}}(\xi) \xi^{1-m}\left( \widetilde{F}(\xi) - \sum^{m-1}_{n=0} F^{n}  \xi^{n} \right),
\end{split}
\end{equation*}
and
\begin{equation*}
\begin{split}
 \frac{1}{\tau^{m}} \sum^{\infty}_{n=1}   \sum^{n-1}_{j=0} s_{j} t_{n+m-1-j}^{m}  \xi^{n}
&  =  {\rho^{m}_{s}}(\xi) \xi^{1-m} \left( \sum^{\infty}_{n=0} t_{n}^{m}  \xi^{n} - \sum^{m-1}_{n=0} t_{n}^{m} \xi^{n} \right) \\
& = {\rho^{m}_{s}}(\xi) \xi^{1-m} \tau^{m}  \left( \gamma_{m}(\xi) - \sum^{m-1}_{n=0} n^{m}  \xi^{n} \right).
\end{split}
\end{equation*}

Let
\begin{equation*}
\begin{split}
G(\xi) :=
&     \frac{\beta Av}{\Gamma(m+1)} {\rho^{m}_{b}}(\xi) \gamma_{m}(\xi) \tau^{m}
  +   \frac{(1-\beta)Av}{\Gamma(m+1)} {\rho^{m}_{s}}(\xi) \xi^{1-m} \tau^{m}  \left( \gamma_{m}(\xi) -  \sum^{m-1}_{n=0} n^{m}  \xi^{n} \right) \\
& + \beta {\rho^{m}_{b}}(\xi) \widetilde{F} (\xi) + (1-\beta) {\rho^{m}_{s}}(\xi) \xi^{1-m}\left( \widetilde{F}(\xi) - \sum^{m-1}_{n=0} F^{n}  \xi^{n} \right).
\end{split}
\end{equation*}
According to the above equations and \eqref{addeq3.4}, this yields
\begin{equation}\label{eq2.11}
\begin{split}
\widetilde{V}(\xi)
& = \left( {\rho_{w}} (\xi) - (\beta + (1-\beta)\xi)A\right)^{-1}  G(\xi) \\
& = \left( \frac{{\rho_{w}} (\xi)}{\beta + (1-\beta)\xi} - A\right)^{-1} \left(\beta + (1-\beta)\xi\right)^{-1} G(\xi).
\end{split}
\end{equation}
From Cauchy's integral formula, and the change of variables $\xi=e^{-z\tau}$, and Cauchy's theorem, this implies \cite{Jin2017Correction}
\begin{equation}\label{DLT}
V^{n}
=\frac{\tau}{2\pi i}\int_{\Gamma^{\tau}_{\theta,\kappa}} \!\!\!e^{zt_n}\left(\frac{{\rho_{w}} (e^{-z\tau})}{\beta + (1-\beta)e^{-z\tau}}- A\right)^{-1} \left(\beta + (1-\beta)e^{-z\tau}\right)^{-1} G(e^{-z\tau})dz
\end{equation}
with $\Gamma^{\tau}_{\theta, \kappa}=\{z\in \Gamma_{\theta, \kappa}: |\Im z|\leq \pi / \tau\}$.
The proof is completed.
\end{proof}
	
%
\section{Preliminaries: A few technical lemmas}\label{Sec4}
We give some lemmas that will be used.
First, we need a few estimates on $\rho_{b}(\xi)$, $\rho_{s}(\xi)$ and $\rho_{w}(\xi)$ in \eqref{CWSBDF}.

\begin{lemma}\label{Lemma3.1}    
Let ${\rho_{b}}$, $\rho_{s}$ and ${\rho^{m}_{b}}$, ${\rho^{m}_{s}}$ with $1\leq m \leq k \leq 7$ be given in \eqref{CWSBDF} and \eqref{CWSBDFkm}, respectively. Then there exists a positive constant $c$ such that
\begin{equation*}
\begin{split}
& | {\rho_{b}}(e^{-z\tau})| \leq c |z|, \quad
   | {\rho_{b}}(e^{-z\tau})-z| \leq c \tau^{k}|z|^{k+1}, \quad
   | {\rho^{m}_{b}}(e^{-z\tau})-z^{m}| \leq c \tau^{k}|z|^{k+m},\\
& | {\rho_{s}}(e^{-z\tau})- e^{-z\tau}z| \leq c \tau^{k}|z|^{k+1}, \quad
| {\rho_{s}}(e^{-z\tau})| \leq c |z|,
   \\
& 
  | {\rho^{m}_{s}}(e^{-z\tau})- e^{-mz\tau}z^{m}| \leq c \tau^{k}|z|^{k+m}, \quad {\forall} z\in \Gamma^{\tau}_{\theta,\kappa},
\end{split}
\end{equation*}
where  $\theta\in (\pi/2,\pi)$ is sufficiently close to $\pi/2$.
\end{lemma}
\begin{proof}
The first two inequalities  of this lemma can be seen in Lemma B.1 of \cite{Jin2017Correction}. Subsequently,
it yields
\begin{equation*}
\begin{split}
\left| {\rho^{m}_{b}}(e^{-z\tau})- z^{m} \right|
& 
 \leq  \left| {\rho_{b}}(e^{-z\tau}) - z  \right|   \cdot \sum_{j=1}^{m} \left| \rho^{m-j}_{b} (e^{-z\tau}) \right|\left|z^{j-1}   \right|
\leq c \tau^{k}|z|^{k+m}.
\end{split}
\end{equation*}

We next estimate  $| {\rho_{s}}(e^{-z\tau})- e^{-z\tau}z| \leq c \tau^{k}|z|^{k+1}$ with $1\leq m \leq k \leq 7$.
According to \eqref{add2.3}, the shifted difference operator ${_{s}\partial_{\tau}} $ with $n\geq k$ can be recast   as
\begin{equation*}
\begin{split}
{_{s}\partial_{\tau}} \varphi^{n}
& = \sum^{k}_{j=1} \frac{\tau^{j-1}}{j}  \partial^{j}_{\tau} \varphi^{n} - \sum^{k}_{j=2} \frac{\tau^{j-1}}{j-1} \partial^{j}_{\tau} \varphi^{n} \quad {\rm with} \quad
\partial_{\tau} \varphi^{n} = \frac{1}{\tau} \left( \varphi^{n} - \varphi^{n-1} \right).
\end{split}
\end{equation*}

We can observe that ${_{s}\partial_{\tau}}$ is an approximation of $\partial_t$ which is accurate of order $k$ at the point $t_{n-1}$.
In fact, by Newton’s interpolation formula we have for $\varphi$ smooth
\begin{equation*}
\varphi(t) = \varphi^{n} + \sum_{j=1}^{k} \frac{(t-t_{n}) \cdots (t-t_{n-j+1})}{\Gamma(j+1)} \partial^{j}_{\tau} \varphi^{n}
+ R_{k}(\varphi;t)
\end{equation*}
with
\begin{equation*}
R_{k}(\varphi;t) = \frac{(t-t_{n}) \cdots (t-t_{n-k})}{\Gamma(k+2)} \varphi^{(k+1)}(\zeta),  \quad \zeta\in[t_{n-k},t_{n}].
\end{equation*}
After differentiation and taking  $t = t_{n-1}$ this shows
\begin{equation*}
\begin{split}
\varphi'(t_{n-1})
& = \partial_{\tau} \varphi^{n} - \sum^{k}_{j=2} \frac{\tau^{j-1}}{(j-1)j} \partial^{j}_{\tau} \varphi^{n} - \frac{\tau^{k}}{k(k+1)} \varphi^{(k+1)}(\zeta) \\
& = {_{s}\partial_{\tau}} \varphi^{n} - \frac{\tau^{k}}{k(k+1)} \varphi^{(k+1)}(\zeta).
\end{split}
\end{equation*}

Let  the characteristic polynomial $ \rho(\xi) = \sum^{k}_{j=0} \overline{s}_{j} \xi^{j}$, where the $\overline{s}_{j}$ are the coefficients in \eqref{CWSBDF}.  Denoting the translation operator $T_{-\tau} \varphi(t) = \varphi(t-\tau)$, it yields
\begin{equation*}
\varphi'(t_{n-1}) =  \tau^{-1} \rho(T_{-\tau}) \varphi^{n} + \mathcal{O}(\tau^{k}).
\end{equation*}
Applying this to $\varphi(t) = e^{t}$ and replacing $\tau$ by $\lambda$, we obtain
\begin{equation*}
\rho(e^{-\lambda}) = e^{-\lambda} \lambda + \mathcal{O}(\lambda^{k+1}).
\end{equation*}
Choosing $\lambda=z\tau$, it implies that
 $| {\rho_{s}}(e^{-z\tau})- e^{-z\tau}z| \leq c \tau^{k}|z|^{k+1}$.
 In particular, it yields $|{\rho_{s}}(e^{-z\tau})| \leq c |z|$, since $|z|\tau \leq \pi/\sin\theta $. Furthermore, we have
\begin{equation*}
\begin{split}
& \left| {\rho^{m}_{s}}(e^{-z\tau})- e^{-mz\tau}z^{m} \right| \\
& \leq  \left| {\rho_{s}}(e^{-z\tau}) - e^{-z\tau}z  \right|  \cdot \sum_{j=1}^{m} \left| \rho^{m-j}_{s} (e^{-z\tau})\right| \left| e^{-(j-1)z\tau}  z^{j-1}   \right|
\leq c \tau^{k}|z|^{k+m}.
\end{split}
\end{equation*}
The proof is completed.
\end{proof}

\begin{lemma}\label{Lemma3.2}    
Let $\rho_{w}$ with $1 \leq m \leq k \leq 7$ be given in \eqref{CWSBDF}. Then there exist the positive constants $c_{1},c_{2}$, $c$, $\varepsilon$, and $\theta \in (\pi/2, \theta_{\varepsilon})$ with  $\theta_{\varepsilon} \in (\pi/2, \pi)$ such that
\begin{equation*}
\begin{split}
&\left| \left(\beta + (1- \beta)e^{-z\tau}\right) \right| \geq c, \quad \beta\geq 3, \\
& c_{1}|z| \leq | {\rho_{w}}(e^{-z\tau})| \leq c_{2}|z|,
\quad \frac{{\rho_{w}}(e^{-z\tau})}{ \beta +(1- \beta )e^{-z\tau}} \in \Sigma_{\pi- \vartheta_{k} +\varepsilon}, \\
& | {\rho_{w}}(e^{-z\tau})- ( \beta +(1- \beta )e^{-z\tau} ) z | \leq c \tau^{k}|z|^{k+1}, \quad
{\forall} z\in \Gamma^{\tau}_{\theta,\kappa},
\end{split}
\end{equation*}
where  $\theta\in (\pi/2,\pi)$ is sufficiently close to $\pi/2$.
\end{lemma}

\begin{proof}
Since $z\tau = r\tau e^{ i\theta}  = r \tau \cos \theta + ir \tau \sin \theta$, it yields
\begin{equation*}
\left| \beta + (1 \! - \! \beta)e^{-z\tau} \right|
\!= \!\left| \beta - (\beta-1) e^{-r \tau \cos \theta}  \cos(r \tau \sin \theta )\!+ \!i (\beta-1) e^{-r \tau \cos \theta}  \sin(r \tau \sin \theta ) \right|.
\end{equation*}
Let $\theta$ is sufficiently close to $\frac{\pi}{2}$ such that $e^{- r \tau \cos \theta} \leq \frac{\beta-\frac{1}{3}}{\beta-1}$. Then we have
\begin{equation*}
\begin{split}
& \left( \beta - (\beta-1) e^{-r \tau \cos \theta}  \cos(r \tau \sin \theta ) \right)^{2} + \left( (\beta-1) e^{-r \tau \cos \theta}  \sin(r \tau \sin \theta ) \right)^{2} \\
& =  \beta^{2} + (\beta-1)^{2} e^{-2 r \tau \cos \theta}  - 2 \beta (\beta-1)  e^{- r \tau \cos \theta}  \cos(r \tau \sin \theta ) \\
& \geq \beta^{2} + (\beta-1)^{2} e^{-2 r \tau \cos \theta}  - 2 \beta (\beta-1)  e^{- r \tau \cos \theta}
= \left( \beta - (\beta-1)  e^{- r \tau \cos \theta} \right)^{2} \geq \frac{1}{9},
\end{split}
\end{equation*}
which implies that $ \left| \beta + (1- \beta)e^{-z\tau} \right| \geq \frac{1}{3}$.

From \eqref{CWSBDF}, we know that the function $\tau \rho_{w}(\xi)/(1 - \xi)$ has no zero in a neighborhood $\mathcal{N}$ of the
unit circle. Since  $ \theta$ is sufficiently close to $\pi/2$ and $e^{-z\tau}$ lies in
the neighborhood $\mathcal{N}$, it yields
\begin{equation*}
c_{1} \leq \frac{| \tau \rho_{w} (e^{-z\tau}) |}{| 1 - e^{-z\tau} |} =  \frac{ | \rho_{w} (e^{-z\tau}) |}{ | (1 - e^{-z\tau})/\tau |}  \leq c_{2}.
\end{equation*}
Since $\widetilde{c}_{1} |z\tau| \leq | 1 - e^{-z\tau} |  \leq \widetilde{c}_{2} |z\tau|$, it leads to $c_{1}|z|\leq | {\rho_{w}}(e^{-z\tau})| \leq c_{2}|z|$.

Let ${\delta}_{\tau} (\xi) := \frac{{\rho_{w}}(\xi)}{ \beta +(1- \beta ) \xi} $ with $ \left| \beta + (1- \beta)\xi \right| \geq \frac{1}{3}$.
When $|\xi| \leq 1$ and $\xi \neq 0$, we have ${\delta}_{\tau} (\xi) \in \Sigma_{\pi- \vartheta_{k} }$ for the $A(\vartheta_{k})$ stable WSBDFk, see Remark \ref{remarkad2.1}. For $|\xi|>1$,  the similar arguments can be performed as  Lemma B.1 of \cite{Jin2017Correction}, we
omit it here.

According to \eqref{CWSBDF} and  Lemma \ref{Lemma3.1}, there exists
\begin{equation*}
\begin{split}
& \left| {\rho_{w}}(e^{-z\tau})  - (\beta+(1-\beta)e^{-z\tau} ) z \right| \\
& = \left| \beta {\rho_{b}}(e^{-z\tau})  + (1-\beta){\rho_{s}}(e^{-z\tau})- \beta z - (1-\beta)e^{-z\tau} z \right|  \\
& \leq \beta \left|  {\rho_{b}}(e^{-z\tau}) - z \right| + (\beta-1) \left| {\rho_{s}}(e^{-z\tau}) - e^{-z\tau} z \right|  \leq c \tau^{k}|z|^{k+1}.
\end{split}
\end{equation*}
The proof is completed.
\end{proof}

\begin{lemma}\label{Lemma3.3}
Let $\gamma_{l}(\xi)=\sum^{\infty}_{n=0}n^{l} \xi^{n}$ with $l=0, 1, 2, \ldots, 2k$. Then there exists a positive constant $c$ such that
\begin{equation*}
\left| \frac{\gamma_{l}(e^{-z\tau})}{\Gamma(l+1)} \tau^{l+1} - \frac{1}{z^{l+1}} \right| \leq
\left\{ \begin{split}
c \tau^{l+1},    \qquad & l=0,1,3,\ldots,2k-1,\\
c \tau^{l+2}|z| , \quad & l=2,4,\ldots,2k.
\end{split}\right.
\end{equation*}
\end{lemma}
\begin{proof}
Let $\overline{\gamma}_{l}(\xi)=\sum^{\infty}_{n=1}n^{l} \xi^{n}$ with $l=0, 1, 2, \ldots, 2k$.
Using Lemma 3.2 of \cite{Shi2023Highorder}, there exists
\begin{equation*}
\left| \frac{\overline{\gamma}_{l}(e^{-z\tau})}{\Gamma(l+1)} \tau^{l+1} - \frac{1}{z^{l+1}} \right| \leq
\left\{ \begin{split}
c \tau^{l+1},    \qquad & l=0,1,3,\ldots,2k-1,\\
c \tau^{l+2}|z| , \quad & l=2,4,\ldots,2k.
\end{split}\right.
\end{equation*}
Thus the desired is obtained for $l\geq 1$.
For $l=0$, we have
\begin{equation*}
\left| \gamma_{0}(e^{-z\tau}) \tau - \frac{1}{z} \right|
= \left| \tau + \overline{\gamma}_{0}(e^{-z\tau}) \tau - \frac{1}{z} \right|
\leq \left| \tau \right| + \left|\overline{\gamma}_{0}(e^{-z\tau}) \tau - \frac{1}{z} \right| \leq
c \tau.
\end{equation*}
The proof is completed.
\end{proof}

\begin{lemma}\label{Lemma3.4}
Let ${\rho^{m}_{b}}(\xi)$ be given in \eqref{CWSBDFkm} with  $1\leq m \leq k \leq 7$ and $\gamma_{l}(\xi)=\sum^{\infty}_{n=0}n^{l} \xi^{n}$  with $0\leq l< k+m$. Then there exists a positive constant $c$ such that
\begin{equation*}
\left| {\rho^{m}_{b}}(e^{-z\tau}) \frac{\gamma_{l}(e^{-z\tau})}{\Gamma(l+1)} \tau^{l+1} -\frac{z^{m}}{z^{l+1}}  \right|
\leq \left\{ \begin{split}
c \tau^{l+1} \left| z \right|^{m} + c \tau^{k}|z|^{k+m-l-1},  \qquad & l=0, 1,3,\ldots,\\
c \tau^{l+2} \left| z \right|^{m+1} + c \tau^{k}|z|^{k+m-l-1}, \quad & l=2,4,\ldots.
\end{split}\right.
\end{equation*}
\end{lemma}

\begin{proof}
Let
\begin{equation*}
{\rho^{m}_{b}}(e^{-z\tau}) \frac{\gamma_{l}(e^{-z\tau})}{\Gamma(l+1)} \tau^{l+1} -\frac{z^{m}}{z^{l+1}} = I_{1} + I_{2}
\end{equation*}
with
\begin{equation*}
I_{1}= {\rho^{m}_{b}}(e^{-z\tau}) \left( \frac{\gamma_{l}(e^{-z\tau})}{\Gamma(l+1)}  \tau^{l+1} - \frac{1}{z^{l+1}} \right) \quad {\rm and} \quad
I_{2}=  \frac{{\rho^{m}_{b}}(e^{-z\tau})-z^m}{z^{l+1}}.
\end{equation*}
According to Lemmas \ref{Lemma3.1} and \ref{Lemma3.3}, this leads to
\begin{equation*}
\left|I_{1} \right|\leq \left\{ \begin{split}
c \tau^{l+1} \left| z \right|^{m} , \qquad & l=0,1,3,\cdots,2k-1,\\
c \tau^{l+2} \left| z \right|^{m+1}, \quad & l=2,4,\cdots,2k;
\end{split}\right.
\end{equation*}
and
\begin{equation*}
\left|I_{2} \right| \leq c \tau^{k}|z|^{k+m-l-1}.
\end{equation*}
The proof is completed.
\end{proof}

\begin{lemma}\label{Lemma3.5}
Let ${\rho^{m}_{s}}(\xi)$ be given in \eqref{CWSBDFkm}   with  $1\leq m \leq k \leq 7$ and $\gamma_{l}(\xi)=\sum^{\infty}_{n=0}n^{l} \xi^{n}$ with $0\leq l< k+m$. Then there exists a positive constant $c$ such that
\begin{equation*}
\begin{split}
& \left| {\rho^{m}_{s}}(e^{-z\tau})  \frac{\gamma_{l}(e^{-z\tau})}{\Gamma(l+1)} \tau^{l+1}  - e^{-mz\tau} \frac{z^{m}}{z^{l+1}} \right| \\
&\quad \leq \left\{ \begin{split}
c \tau^{l+1} \left| z \right|^{m} + c \tau^{k}|z|^{k+m-l-1},  \qquad & l=0,1,3,\ldots,\\
c \tau^{l+2} \left| z \right|^{m+1} + c \tau^{k}|z|^{k+m-l-1}, \quad & l=2,4,\ldots.
\end{split}\right.
\end{split}
\end{equation*}
\end{lemma}

\begin{proof}
We consider
\begin{equation*}
\begin{split}
 {\rho^{m}_{s}}(e^{-z\tau})   \frac{\gamma_{l}(e^{-z\tau})}{\Gamma(l+1)} \tau^{l+1}  - e^{-mz\tau} \frac{z^{m}}{z^{l+1}}  = I_{1} + I_{2}
\end{split}
\end{equation*}
with
\begin{equation*}
\begin{split}
I_{1}  &=  {\rho^{m}_{s}}(e^{-z\tau})  \left( \frac{\gamma_{l}(e^{-z\tau})}{\Gamma(l+1)} \tau^{l+1} -  \frac{1}{z^{l+1}} \right) \quad {\rm and} \quad
I_{2} =  \left( {\rho^{m}_{s}}(e^{-z\tau})  - e^{-mz\tau}z^{m} \right) \frac{1}{z^{l+1}}.
\end{split}
\end{equation*}
According to Lemmas \ref{Lemma3.1} and \ref{Lemma3.3}, we obtain
\begin{equation*}
\left| I_{1} \right| \leq
\left\{ \begin{split}
c \tau^{l+1}|z|^{m},    \qquad & l=0,1,3,\ldots,\\
c \tau^{l+2}|z|^{m+1},   \quad & l=2,4,\ldots,
\end{split}\right.
\quad  {\rm and} \quad
\left| I_{2} \right| \leq c \tau^{k}|z|^{k+m-l-1}.
\end{equation*}
The proof is completed.
\end{proof}

From Lemmas \ref{Lemma3.1}--\ref{Lemma3.5}, we have the following results, which will be used in the global convergence analysis.
\begin{lemma}\label{Lemma3.6}
Let ${\rho^{m}_{b}}$, ${\rho^{m}_{s}}$ and $\rho_{w}$  be given in \eqref{CWSBDFkm} and \eqref{CWSBDF}, respectively. Let $\gamma_{l}(\xi)=\sum^{\infty}_{n=0}n^{l} \xi^{n}$ with $0 \leq l < k+m$ and $1\leq m \leq k \leq 7$. Then there exists a positive constant $c$ such that
\begin{equation*}
\left\| K(z) \right\|\leq
c \tau^{l+1} \left| z \right|^{m-1} + c \tau^{k}|z|^{k+m-l-2}.
\end{equation*}
Here the kernel is
\begin{equation*}
K(z)= \left(\frac{{\rho_{w}} (e^{-z\tau})}{\beta + (1-\beta)e^{-z\tau}}- A\right)^{-1} \frac{ G_{1} (e^{-z\tau})}{\beta + (1-\beta)e^{-z\tau}} - (z - A)^{-1}  \frac{z^{m}}{z^{l+1}}
\end{equation*}
with
\begin{equation}\label{eq4.2}
\begin{split}
G_{1}(\xi)
& =  \beta  {\rho^{m}_{b}}(\xi) \frac{\gamma_{l}(\xi)}{\Gamma(l+1)} \tau^{l+1} + (1-\beta) {\rho^{m}_{s}}(\xi) \xi^{1-m}  \frac{\gamma_{l}(\xi)}{\Gamma(l+1)} \tau^{l+1} \\
& \quad - (1-\beta) \tau^{l+1}  {\rho^{m}_{s}}(\xi) \xi^{1-m} \sum^{m-1}_{n=0} n^{l}  \xi^{n}.
\end{split}
\end{equation}
\end{lemma}
	
\begin{proof}
Let $ K(z) = I_{1} + I_{2} $ with
\begin{equation}\label{eq4.1}
\begin{split}
I_{1}  & = \left(\frac{{\rho_{w}} (e^{-z\tau})}{\beta + (1-\beta)e^{-z\tau}}- A\right)^{-1}  \frac{ G_{1} (e^{-z\tau}) - \left(\beta + (1-\beta)e^{-z\tau}\right) z^{m-l-1}}{\beta + (1-\beta)e^{-z\tau}}, \\
I_{2} & = \left[ \left(\frac{{\rho_{w}} (e^{-z\tau})}{\beta + (1-\beta)e^{-z\tau}}- A\right)^{-1} - (z - A)^{-1}  \right] z^{m-l-1}.
\end{split}
\end{equation}
The resolvent estimate \eqref{resolvent estimate} and Lemma \ref{Lemma3.2} imply directly
\begin{equation}\label{discrete fractional resolvent estimate}
\left\| \left(\frac{{\rho_{w}} (e^{-z\tau})}{\beta + (1-\beta)e^{-z\tau}}- A\right)^{-1} \right\| \leq c |z|^{-1}.
\end{equation}
From \eqref{eq4.2}, we obtain
\begin{equation*}
\begin{split}
& G_{1} (e^{-z\tau}) - \left(\beta + (1-\beta)e^{-z\tau}\right) z^{m-l-1} \\
&  =  \beta \left( {\rho^{m}_{b}}(e^{-z\tau}) \frac{\gamma_{l}(e^{-z\tau})}{\Gamma(l+1)} \tau^{l+1} -  z^{m-l-1}\right) \\
& \quad + (1-\beta) e^{-(1-m)z\tau} \left(  {\rho^{m}_{s}}(e^{-z\tau})   \frac{\gamma_{l}(e^{-z\tau})}{\Gamma(l+1)} \tau^{l+1}  - e^{-mz\tau}z^{m-l-1}  \right) \\
& \quad  -  (1-\beta) \tau^{l+1}  {\rho^{m}_{s}}(e^{-z\tau}) e^{-(1-m)z\tau} \sum^{m-1}_{n=0} n^{l}  e^{-nz\tau}.
\end{split}
\end{equation*}
According to Lemmas \ref{Lemma3.4}, \ref{Lemma3.5} and \ref{Lemma3.1}, it yields
\begin{equation}\label{eq4.4}
\left\| I_{1} \right\|\leq \left\{ \begin{split}
c \tau^{l+1} \left| z \right|^{m-1} + c \tau^{k}|z|^{k+m-l-2},  \qquad\qquad\qquad ~ \, & l=0,1,3,\ldots,\\
c \tau^{l+1} \left| z \right|^{m-1} + c \tau^{l+2} \left| z \right|^{m} + c \tau^{k}|z|^{k+m-l-2},   \quad & l=2,4,\ldots.
\end{split}\right.
\end{equation}
		
On the other hand, using \eqref{resolvent estimate}, \eqref{discrete fractional resolvent estimate}, Lemma \ref{Lemma3.2} and the identity
\begin{equation*}
\begin{split}
& \left(\frac{{\rho_{w}} (e^{-z\tau})}{\beta + (1-\beta)e^{-z\tau}}- A\right)^{-1} - (z - A)^{-1} \\
& = \left(\frac{{\rho_{w}} (e^{-z\tau})}{\beta + (1-\beta)e^{-z\tau}}- A\right)^{-1} (z - A)^{-1} \left(z - \frac{{\rho_{w}} (e^{-z\tau})}{\beta + (1-\beta)e^{-z\tau}} \right),
\end{split}
\end{equation*}
we estimate $I_{2}$ as
\begin{equation*}
\| I_{2} \|  \leq c \tau^{k} |z|^{k + 1}  c |z|^{-1} c |z|^{-1}  |z|^{m-l-1} \leq c \tau^{k} |z|^{k+m-l-2}.
\end{equation*}
The proof is completed.
\end{proof}

\begin{lemma}\label{Lemma3.7}
Let ${\rho^{m}_{b}}$, ${\rho^{m}_{s}}$ and ${\rho_{w}}$  be given in \eqref{CWSBDFkm} and \eqref{CWSBDF}, respectively. Let $\gamma_{m}(\xi)=\sum^{\infty}_{n=0}n^{m} \xi^{n}$ with $1\leq m \leq k \leq 7$. Then there exists a positive constant $c$ such that
\begin{equation*}
\left\|\mathcal{K}(z)  \right\|\leq
c \tau^{m+1} \left| z \right|^{m} + c \tau^{k}|z|^{k-1}.  
\end{equation*}
Here the kernel is
\begin{equation*}
\mathcal{K} (z)= \left(\frac{{\rho_{w}} (e^{-z\tau})}{\beta + (1-\beta)e^{-z\tau}}- A\right)^{-1} \left(\beta + (1-\beta)e^{-z\tau}\right)^{-1} G_{2} (e^{-z\tau}) A - (z - A)^{-1}  z^{-1} A
\end{equation*}
with
\begin{equation}\label{eq4.3}
\begin{split}
G_{2}(\xi)
& =  \beta  {\rho^{m}_{b}}(\xi) \frac{\gamma_{m}(\xi)}{\Gamma(m+1)} \tau^{m+1} + (1-\beta) {\rho^{m}_{s}}(\xi) \xi^{1-m}  \frac{\gamma_{m}(\xi)}{\Gamma(m+1)} \tau^{m+1} \\
& \quad - (1-\beta) \tau^{m+1}  {\rho^{m}_{s}}(\xi) \xi^{1-m} \sum^{m-1}_{n=0} n^{m}  \xi^{n}.
\end{split}
\end{equation}
\end{lemma}
	
\begin{proof}
Since $(z - A)^{-1} z^{-1} A =    (z - A)^{-1} - z^{-1} $ and
\begin{equation*}
\begin{split}
& \left(\frac{{\rho_{w}} (e^{-z\tau})}{\beta + (1-\beta)e^{-z\tau}}- A\right)^{-1} \left(\beta + (1-\beta)e^{-z\tau}\right)^{-1}  G_{2}(e^{-z\tau}) A \\
& = \left(\frac{{\rho_{w}} (e^{-z\tau})}{\beta + (1-\beta)e^{-z\tau}}- A\right)^{-1} \left(\beta + (1-\beta)e^{-z\tau}\right)^{-1}  G_{2}(e^{-z\tau}) \frac{{\rho_{w}} (e^{-z\tau})}{\beta + (1-\beta)e^{-z\tau}} \\
& \quad - \left(\beta + (1-\beta)e^{-z\tau}\right)^{-1} G_{2}(e^{-z\tau}).
\end{split}
\end{equation*}
Then we can split $\mathcal{K}(z)$ as
\begin{equation*}
\mathcal{K}(z) = I_{1} + I_{2} + I_{3}
\end{equation*}
with
\begin{equation*}
\begin{split}
I_{1} =& \left[ \left(\frac{{\rho_{w}} (e^{-z\tau})}{\beta + (1-\beta)e^{-z\tau}}- A\right)^{-1} \! \frac{ G_{2} (e^{-z\tau})}{\beta + (1-\beta)e^{-z\tau}} - (z - A)^{-1} z^{-1} \right] \frac{{\rho_{w}} (e^{-z\tau})}{\beta + (1-\beta)e^{-z\tau}}, \\
I_{2} =& (z - A)^{-1} z^{-1} \left[ \frac{{\rho_{w}} (e^{-z\tau})}{\beta + (1  -  \beta)e^{-z\tau}} - z \right], \\
I_{3} =& z^{-1}  - \left(\beta + (1-\beta)e^{-z\tau}\right)^{-1} G_{2} (e^{-z\tau}) = \frac{ \left(\beta + (1-\beta)e^{-z\tau} \right) z^{-1}   -  G_{2} (e^{-z\tau})}{\beta + (1-\beta)e^{-z\tau}}.
\end{split}
\end{equation*}

From Lemma \ref{Lemma3.6} with $l=m$ and Lemma \ref{Lemma3.2}, we estimate $I_{1}$ as follows
\begin{equation*}
\left\|I_{1} \right\| \leq  \left\| K(z) \right\|
\left| \frac{{\rho_{w}} (e^{-z\tau})}{\beta + (1-\beta)e^{-z\tau}} \right| \leq
c \tau^{m+1} \left| z \right|^{m} + c \tau^{k}|z|^{k-1}.
\end{equation*}
Moreover, using \eqref{resolvent estimate} and Lemma \ref{Lemma3.2}, we have
\begin{equation*}
\left\| I_{2}  \right\| \leq c |z|^{-2} \tau^{k}|z|^{k+1}   \leq   c \tau^{k} \left| z \right|^{k-1}.
\end{equation*}
On the other hand,  according to \eqref{eq4.1} and \eqref{eq4.4} with $l=m$, it yields
\begin{equation*}
\left\|I_{3} \right\| \leq
c \tau^{m+1} \left| z \right|^{m} + c \tau^{k} |z|^{k-1}.
\end{equation*}
The proof is completed.
\end{proof}

\section{Error analysis}\label{Sec5}
We first consider the homogeneous problem for \eqref{PQ}.
\begin{theorem}\label{Theorem3.2}
Let $V(t_{n})$ and $V^{n}$ be the solutions of \eqref{RPQ} and \eqref{IDWSBDF}, respectively.
Let $v\in L^{2}(\Omega)$ and $f(t)=0$.
Then the following error estimate holds for any $t_n>0$:
\begin{equation*}
\left\| V(t_{n}) -   V^{n}  \right\|_{L^2(\Omega)} \leq \left( c \tau^{m+1} t_{n}^{-m-1} + c\tau^{k} t_{n}^{-k} \right) \left\| v \right\|_{L^2(\Omega)}
\end{equation*}
with $1\leq m \leq k \leq 7$.
\end{theorem}
\begin{proof}
From \eqref{LT} and Lemma \ref{Lemma2.1}, we have
\begin{equation*}
V(t_{n}) =  \frac{1}{2\pi i} \int_{\Gamma_{\theta, \kappa}} e^{zt_{n}} (z - A)^{-1}  z^{-1}  Av dz
\end{equation*}
and
\begin{equation*}
V^{n}=\frac{1}{2\pi i}\int_{\Gamma^{\tau}_{\theta,\kappa}} e^{zt_n} \left( \frac{{\rho_{w}} (e^{-z\tau})}{\beta + (1-\beta)e^{-z\tau}} - A\right)^{-1} \left(\beta + (1-\beta)e^{-z\tau}\right)^{-1} G_{2} (e^{-z\tau}) Av dz,
\end{equation*}
where $G_{2}(\xi)$ is given in \eqref{eq4.3} and $\gamma_{m}(\xi)=\sum^{\infty}_{n=1}n^{m} \xi^{n}$.
Then we obtain
\begin{equation*}
V(t_{n}) - V^{n}=J_1 + J_2
\end{equation*}
with
\begin{equation*}
J_{1} = \frac{1}{2\pi i}\int_{\Gamma^{\tau}_{\theta,\kappa}} -e^{zt_{n}} \mathcal{K} (z) v dz \quad {\rm and} \quad
J_{2} =\frac{1}{2\pi i}\int_{\Gamma_{\theta,\kappa}\setminus\Gamma^{\tau}_{\theta,\kappa}}{e^{zt_{n}}\left(z - A\right)^{-1} z^{-1} Av  }dz
\end{equation*}
with $\mathcal{K} (z)$ in Lemma \ref{Lemma3.7}.

According to the triangle inequality and Lemma \ref{Lemma3.7}, this yields
\begin{equation*}
\begin{split}
\| J_1 \|_{L^2(\Omega)}
&  \leq c  \int^{\frac{\pi}{\tau\sin\theta}}_{\kappa} e^{rt_{n}\cos\theta} \left(\tau^{m+1}  r^{m} +  \tau^{k} r^{k-1} \right) dr \left\|  v \right\|_{L^2(\Omega)} \\
& \quad + c \int^{\theta}_{-\theta}e^{\kappa t_{n} \cos\psi} \left(\tau^{m+1}  \kappa^{m+1} +  \tau^{k} \kappa^{k} \right) d\psi  \left\|  v \right\|_{L^2(\Omega)} \\
&  \leq  \left( c\tau^{m+1}t_{n}^{-m-1} + c\tau^{k}t_{n}^{-k} \right) \left\|  v \right\|_{L^2(\Omega)}, \quad m=1,2,\ldots,7,
\end{split}
\end{equation*}
where  we use
\begin{equation}\label{eq3.3}
\int^{\frac{\pi}{\tau\sin\theta}}_{\kappa} e^{rt_{n}\cos\theta} r^{k-1}dr \leq c t_n^{-k} \quad {\rm and} \quad
\int^{\theta}_{-\theta}e^{\kappa t_{n} \cos\psi} \kappa^{k} d \psi \leq c t_n^{-k}.
\end{equation}
Using $(z - A)^{-1} z^{-1} A =   (z - A)^{-1} - z^{-1}$ and resolvent estimate \eqref{resolvent estimate}, we have
\begin{equation*}
\begin{split}
\|J_2 \|_{L^2(\Omega)}
& \leq c  \int^{\infty}_{\frac{\pi}{\tau\sin\theta}} e^{rt_{n}\cos\theta}r^{-1}dr \left\| v \right\|_{L^2(\Omega)}
 \leq c\tau^{k}  \int^{\infty}_{\frac{\pi}{\tau\sin\theta}} e^{rt_{n}\cos\theta}r^{k-1}dr  \left\| v \right\|_{L^2(\Omega)} \\
& \leq  c\tau^{k}t_{n}^{-k} \left\| v \right\|_{L^2(\Omega)}.
\end{split}
\end{equation*}
Here we use $1\leq (\frac{\sin \theta}{\pi})^{k} \tau^{k} r^{k}$ for $r\geq \frac{\pi}{\tau \sin \theta}$. The proof is completed.
\end{proof}

Now, we consider the general source function $f(t)$.
Using  Taylor series  expansion  with the remainder term in integral form, it yields
\begin{equation*}
\begin{split}
f(t)  &= \sum_{l=0}^{k-1}  \frac{t^l}{\Gamma(l+1)} \partial_{t}^{l} f(0)  + \frac{t^{k-1}}{\Gamma(k)} \ast \partial_{t}^{k} f(t),\\
F(t)
& =  \sum_{l=0}^{k+m-1}  \frac{t^l}{\Gamma(l+1)} \partial_{t}^{l} F(0)  + \frac{t^{k+m-1}}{\Gamma(k+m)}  \ast  \partial_{t}^{k+m} F(t).
\end{split}
\end{equation*}
Using the identity $f(t) = \partial_t^mF(t)$ in \eqref{smoothing1}, we obtain
\begin{equation}\label{gs3.3}
\begin{split}
F(t) =  \sum_{l=0}^{k+m-1}  \frac{t^l}{\Gamma(l+1)} \partial_{t}^{l-m} f(0)  + \frac{t^{k+m-1}}{\Gamma(k+m)}  \ast  \partial_{t}^{k} f(t)
\end{split}
\end{equation}
with $\partial_{t}^{-j} f(0)=J^{j}f(0)=0$, $j > 0$.
Then we obtain the following results.
\begin{lemma}\label{lemma3.9}
Let $V(t_{n})$ and $V^{n}$ be the solutions of \eqref{RPQ} and \eqref{IDWSBDF}, respectively.
Let $v=0$ and $F(t): = \frac{t^l}{\Gamma(l+1)} \partial_{t}^{l-m} f(0)$ with $l=0,1,2, \ldots, k+m-1$, $1\leq m \leq k\leq 7$.
Then the following error estimate holds for any $t_n>0$:
\begin{equation*}
\left\| V(t_{n})  -  V^{n}  \right\|_{L^2(\Omega)} \leq \left( c\tau^{l+1}t_{n}^{-m} + c\tau^{k}t_{n}^{l+1-k-m} \right) \left\| \partial_{t}^{l-m} f(0) \right\|_{L^2(\Omega)}.
\end{equation*}
\end{lemma}
\begin{proof}
From \eqref{LT} and Lemma \ref{Lemma2.1}, we have
\begin{equation*}
V(t_{n})=\frac{1}{2\pi i}\int_{\Gamma_{\theta,\kappa}}{e^{zt_{n}}(z^{\alpha}-A)^{-1} z^{m} z^{-l-1} \partial_{t}^{l-m} f(0) }dz,~~l=0,1,2, \ldots, k+m-1,
\end{equation*}
and
\begin{equation*}
V^{n}=  \frac{1}{2\pi i}\int_{\Gamma^{\tau}_{\theta,\kappa}} e^{zt_n} \left( \frac{{\rho_{w}} (\xi)}{\beta + (1-\beta)e^{-z\tau}} - A\right)^{-1}   \frac{G_{1} (e^{-z\tau})}{\beta +  (1 -  \beta)e^{-z\tau}}  \partial_{t}^{l-m} f(0) dz,
\end{equation*}
where $G_{1}(\xi)$ is given in \eqref{eq4.2} 
and $\gamma_{l}(\xi)=\sum^{\infty}_{n=0}n^{l} \xi^{n}$.
Then we have
\begin{equation*}
V(t_{n})-V^{n}=J_{1} + J_{2},
\end{equation*}
where
\begin{equation*}
\begin{split}
J_{1} &
= \frac{1}{2\pi i}\int_{\Gamma^{\tau}_{\theta,\kappa}} -e^{zt_{n}} K(z)  \partial_{t}^{l-m} f(0) dz, \\
J_{2} & =\frac{1}{2\pi i}\int_{\Gamma_{\theta,\kappa}\setminus\Gamma^{\tau}_{\theta,\kappa}}{e^{zt_{n}}\left(z - A\right)^{-1}  \frac{z^{m}}{z^{j+1}} \partial_{t}^{l-m} f(0) }dz
\end{split}
\end{equation*}
with $K(z)$ in Lemma \ref{Lemma3.6}.
Hence, using \eqref{resolvent estimate}, \eqref{eq3.3} and Lemma \ref{Lemma3.6}, we have
\begin{equation*}
\begin{split}
\| J_1 \|_{L^2(\Omega)}
&  \leq c  \int^{\frac{\pi}{\tau\sin\theta}}_{\kappa} e^{rt_{n}\cos\theta} \left(\tau^{l+1}  r^{m-1} +  \tau^{k} r^{k+m-l-2} \right) dr \left\| \partial_{t}^{l-m} f(0) \right\|_{L^2(\Omega)} \\
& \quad + c \int^{\theta}_{-\theta}e^{\kappa t_{n} \cos\psi} \left(\tau^{l+1}  \kappa^{m} +  \tau^{k} \kappa^{k+m-l-1} \right) d\psi  \left\| \partial_{t}^{l-m} f(0) \right\|_{L^2(\Omega)} \\
&  \leq  \left( c\tau^{l+1}t_{n}^{-m} + c\tau^{k}t_{n}^{l+1-k-m} \right) \left\| \partial_{t}^{l-m} f(0) \right\|_{L^2(\Omega)}. 
\end{split}
\end{equation*}
On the other hand, using \eqref{resolvent estimate}, it leads to
\begin{equation*}
\begin{split}
\|J_2 \|_{L^2(\Omega)}
&\leq c  \int^{\infty}_{\frac{\pi}{\tau\sin\theta}} e^{rt_{n}\cos\theta}r^{m-l-2}dr \left\| \partial_{t}^{l-m} f(0) \right\|_{L^2(\Omega)}\\
&\leq c\tau^{k}  \int^{\infty}_{\frac{\pi}{\tau\sin\theta}} e^{rt_{n}\cos\theta}r^{k+m-l-2}dr  \left\| \partial_{t}^{l-m} f(0) \right\|_{L^2(\Omega)}\\
&\leq  c\tau^{k}t_{n}^{l+1-k-m} \left\| \partial_{t}^{l-m} f(0) \right\|_{L^2(\Omega)}.
\end{split}
\end{equation*}
Here we use $1\leq (\frac{\sin \theta}{\pi})^{k} \tau^{k} r^{k}$ with $r\geq \frac{\pi}{\tau \sin \theta}$.
The proof is completed.
\end{proof}

\begin{lemma}\label{lemma3.10}
Let $V(t_{n})$ and $V^{n}$ be the solutions of \eqref{RPQ} and \eqref{IDWSBDF}, respectively.
Let $v=0$, $F(t): = \frac{t^{k+m-1}}{\Gamma(k+m)} \ast \partial_{t}^{k} f(t)$, $1 \leq  m \leq k \leq  7$, and $\int_{0}^{t} \| \partial_{t}^{k} f(s) \|_{L^2(\Omega)} ds  < \infty$. Then the following error estimate holds for any $t_n>0$:
\begin{equation*}
\left\|V(t_{n})-V^{n}\right\|_{L^2(\Omega)} \leq c\tau^{k} \int_{0}^{t_{n}}  \left\|\partial_{s}^{k} f(s) \right\|_{L^2(\Omega)}ds.
\end{equation*}
\end{lemma}

\begin{proof}
According to \eqref{LT}, 
we have
\begin{equation}\label{nas3.6}
V(t_{n}) = \left( \mathscr{E}(t)\ast F(t) \right)(t_{n})
= \left(\left(\mathscr{E}(t)\ast \frac{t^{k+m-1}}{\Gamma(k+m)} \right) \ast  \partial_{t}^{k} f(t) \right)(t_{n})
\end{equation}
with
\begin{equation}\label{nas3.007}
\mathscr{E}(t)= \frac{1}{2\pi i} \int_{\Gamma_{\theta,\kappa}} e^{zt}(z - A)^{-1} z^{m} dz.
\end{equation}

From \eqref{eq2.11}, we denote
\begin{equation*}
\sum^{\infty}_{n=0} {_{b}\mathscr{E}^{n}_{\tau}}\xi^{n}
= \widetilde{{_{b}\mathscr{E_{\tau}}}}(\xi)
:=\left( \frac{{\rho_{w}} (\xi)}{\beta + (1-\beta)\xi} - A\right)^{-1} \left(\beta + (1-\beta)\xi\right)^{-1} \beta {\rho^{m}_{b}}(\xi),
\end{equation*}
and
\begin{equation*}
\sum^{\infty}_{n=0} {_{s}\mathscr{E}^{n}_{\tau}} \xi^{n}
=\widetilde{{_{s}\mathscr{E_{\tau}}}}(\xi)
:=\left( \frac{{\rho_{w}} (\xi)}{\beta + (1-\beta)\xi} - A\right)^{-1} \left(\beta + (1-\beta)\xi\right)^{-1} (1-\beta) {\rho^{m}_{s}}(\xi) \xi.
\end{equation*}
Moreover, this yields
\begin{equation*}
\begin{split}
\widetilde{V}(\xi)
& = \widetilde{{_{b}\mathscr{E_{\tau}}}}(\xi) \sum^{\infty}_{j=0} F^{j} \xi^{j} + \widetilde{{_{s}\mathscr{E}_{\tau}}}(\xi) \sum^{\infty}_{j=0} F^{j+m} \xi^{j}
 = \sum^{\infty}_{n=0} \sum^{\infty}_{j=0} \left( {_{b}\mathscr{E}^{n}_{\tau}} F^{j} +  {_{s}\mathscr{E}^{n}_{\tau}} F^{j+m}\right) \xi^{n+j} \\
& =\sum^{\infty}_{j=0} \sum^{\infty}_{n=j} \left( {_{b}\mathscr{E}^{n-j}_{\tau}} F^{j}  +  {_{s}\mathscr{E}^{n-j}_{\tau}} F^{j+m} \right) \xi^{n}
 =\sum^{\infty}_{n=0} \sum^{n}_{j=0} \left( {_{b}\mathscr{E}^{n-j}_{\tau}} F^{j} + {_{s}\mathscr{E}^{n-j}_{\tau}} F^{j+m} \right) \xi^{n} \\
& 
=\sum^{\infty}_{n=0} V^{n} \xi^{n}
\end{split}
\end{equation*}
with
\begin{equation}\label{addv1}
V^{n}=\sum^{n}_{j=0} \left( {_{b}\mathscr{E}^{n-j}_{\tau}} F^{j} + {_{s}\mathscr{E}^{n-j}_{\tau}} F^{j+m} \right)
=\sum^{n}_{j=0} \left( {_{b}\mathscr{E}^{n-j}_{\tau}} F(t_{j}) + {_{s}\mathscr{E}^{n-j}_{\tau}} F(t_{j+m}) \right).
\end{equation}
			
Using Cauchy's integral formula, the change of variables $\xi=e^{-z\tau}$ and Cauchy's theorem, we get the representation of the ${_{b}\mathscr{E}^{n}_{\tau}}$ and ${_{s}\mathscr{E}^{n}_{\tau}}$ respectively as follows: 
\begin{equation*}
\begin{split}
{_{b}\mathscr{E}^{n}_{\tau}}
& =\frac{1}{2\pi i}\int_{|\xi|=\varrho}{\xi^{-n-1} \widetilde{{_{b}\mathscr{E_{\tau}}}} (\xi)}d\xi
=\frac{\tau}{2\pi i}\int_{\Gamma^{\tau}}{e^{zt_n} \widetilde{{_{b}\mathscr{E_{\tau}}}} (e^{-z\tau})}dz\\
& =\frac{\tau}{2\pi i}\int_{\Gamma^{\tau}_{\theta,\kappa}} {e^{zt_n} \left( \frac{{\rho_{w}} (e^{-z\tau})}{\beta + (1-\beta)e^{-z\tau}} - A\right)^{-1} \left(\beta + (1-\beta) e^{-z\tau} \right)^{-1} \beta {\rho^{m}_{b}}(e^{-z\tau}) }dz,
\end{split}
\end{equation*}
and
\begin{equation*}
\begin{split}
{_{s}\mathscr{E}^{n}_{\tau}}
& =\frac{1}{2\pi i}\int_{|\xi|=\varrho}{\xi^{-n-1} \widetilde{{_{s}\mathscr{E_{\tau}}}}(\xi)}d\xi =\frac{\tau}{2\pi i}\int_{\Gamma^{\tau}}{e^{zt_n} \widetilde{{_{s}\mathscr{E_{\tau}}}} (e^{-z\tau})}dz\\
& =\frac{\tau}{2\pi i}\int_{\Gamma^{\tau}_{\theta,\kappa}} {e^{zt_n} \left( \frac{{\rho_{w}} (e^{-z\tau})}{\beta + (1-\beta) e^{-z\tau}} - A\right)^{-1} \frac{(1-\beta) {\rho^{m}_{s}}(e^{-z\tau}) e^{-z\tau}}{\beta + (1-\beta)e^{-z\tau}}   }dz,
\end{split}
\end{equation*}
where $\varrho\in(0,1)$ and $\Gamma^{\tau} := \{ z = -\ln(\varrho)/\tau + i y: y\in \mathbb{R} ~{\rm and}~ |y|\leq \pi/\tau  \}$ in \cite{Jin2017Correction}.

According to  \eqref{discrete fractional resolvent estimate} and Lemmas \ref{Lemma3.1}, \ref{Lemma3.2}, 
this implies
\begin{equation}\label{3.0002}
\left\| {_{b}\mathscr{E}^{n}_{\tau}} \right\| \leq c \tau \left( \int^{\frac{\pi}{\tau\sin\theta}}_{\kappa} e^{rt_{n}\cos\theta} r^{m-1}dr +\int^{\theta}_{-\theta}e^{\kappa t_{n}\cos\psi} \kappa^{m}  d\psi\right)
\leq c \tau t_{n}^{-m}, 
\end{equation}
and
\begin{equation}\label{eq3.9}
\left\| {_{s}\mathscr{E}^{n}_{\tau}} \right\| \leq c \tau \left( \int^{\frac{\pi}{\tau\sin\theta}}_{\kappa} e^{rt_{n}\cos\theta} r^{m-1}dr +\int^{\theta}_{-\theta}e^{\kappa t_{n}\cos\psi} \kappa^{m}  d\psi\right) \leq c \tau t_{n}^{-m}.
\end{equation}

Let $ {_{b}\mathscr{E}_{\tau}}(t):=\sum^{\infty}_{n=0} {_{b}\mathscr{E}^{n}_{\tau}} \delta_{t_{n}}(t)$ and $ {_{s}\mathscr{E}_{\tau}}(t):=\sum^{\infty}_{n=0} {_{s}\mathscr{E}^{n}_{\tau}} \delta_{t_{n}}(t+t_{m})$, with $\delta_{t_{n}}$ being the Dirac delta function at $t_{n}$.
Then we obtain
\begin{equation}\label{nas3.8}
\begin{split}
& ( {_{b}\mathscr{E}_{\tau}}(t)\ast F(t))(t_{n}) + ( {_{s}\mathscr{E}_{\tau}} (t)\ast F(t))(t_{n}) \\
& = \left(\sum^{\infty}_{j=0} {_{b}\mathscr{E}^{j}_{\tau}} \delta_{t_{j}}(t) \ast F(t) \right)(t_{n}) + \left(\sum^{\infty}_{j=0} {_{s}\mathscr{E}^{j}_{\tau}} \delta_{t_{j}}(t+t_{m}) \ast F(t) \right)(t_{n}) \\
& = \sum^{n}_{j=0} {_{b}\mathscr{E}^{j}_{\tau}} F(t_{n}-t_{j}) + \sum^{n}_{j=0} {_{s}\mathscr{E}^{j}_{\tau}} F(t_{n}+t_{m}-t_{j})  \\
& = \sum^{n}_{j=0} {_{b}\mathscr{E}^{n-j}_{\tau}} F(t_{j}) + \sum^{n}_{j=0} {_{s}\mathscr{E}^{n-j}_{\tau}} F(t_{j}+t_{m}) = V^{n},
\end{split}
\end{equation}
where $V^n$ is given in \eqref{addv1}.
In particular, from \eqref{nas3.8}, there exist
\begin{equation*}
\begin{split}
\widetilde{({_{b}\mathscr{E}_{\tau}}\ast t^{l})}(\xi)
& = \sum^{\infty}_{n=0} \sum^{n}_{j=0} {_{b}\mathscr{E}^{n-j}_{\tau}} t^{l}_{j}\xi^{n}  =\sum^{\infty}_{j=0} \sum^{\infty}_{n=j} {_{b}\mathscr{E}^{n-j}_{\tau}} t^{l}_{j}\xi^{n}
  = \sum^{\infty}_{j=0} \sum^{\infty}_{n=0} {_{b}\mathscr{E}^{n}_{\tau}} t^{l}_{j}\xi^{n+j}\\
& = \sum^{\infty}_{n=0} {_{b}\mathscr{E}^{n}_{\tau}} \xi^{n}\sum^{\infty}_{j=0}t^{l}_{j}\xi^{j}  =\widetilde{{_{b}\mathscr{E}_{\tau}}}(\xi) \tau^{l} \sum^{\infty}_{j=0}j^{l}\xi^{j}
  = \widetilde{{_{b}\mathscr{E}_{\tau}}}(\xi) \tau^{l} \gamma_{l}(\xi)
\end{split}
\end{equation*}
with $\gamma_{l}(\xi)=\sum_{n=0}^{\infty} n^{l} \xi^{n}$, $0 \leq l \leq k+m-1$, and
\begin{equation*}
\begin{split}
\widetilde{({_{s}\mathscr{E}_{\tau}}\ast t^{l})}(\xi)
& = \sum^{\infty}_{n=0} \sum^{n}_{j=0} {_{s}\mathscr{E}^{n-j}_{\tau}} t^{l}_{j+m}\xi^{n}  =\sum^{\infty}_{j=0} \sum^{\infty}_{n=j} {_{s}\mathscr{E}^{n-j}_{\tau}} t^{l}_{j+m}\xi^{n} \\
& = \sum^{\infty}_{j=0} \sum^{\infty}_{n=0} {_{s}\mathscr{E}^{n}_{\tau}} t^{l}_{j+m}\xi^{n+j} = \sum^{\infty}_{n=0} {_{s}\mathscr{E}^{n}_{\tau}} \xi^{n}\sum^{\infty}_{j=0}t^{l}_{j+m}\xi^{j} \\
& = \widetilde{{_{s}\mathscr{E_{\tau}}}}(\xi) \tau^{l} \sum^{\infty}_{j=0}(j+m)^{l}\xi^{j}
  = \widetilde{{_{s}\mathscr{E}_{\tau}}}(\xi) \tau^{l} \xi^{-m} \left( \gamma_{l}(\xi) - \sum^{m-1}_{j=0} j^{l}\xi^{j}\right).
\end{split}
\end{equation*}
			
From  \eqref{nas3.6}, \eqref{nas3.8}, and Lemma \ref{lemma3.9}, we have
\begin{equation}\label{aadd3.8}
\left\| \left( \left( {_{b}\mathscr{E}_{\tau}} + {_{s}\mathscr{E}_{\tau}}  - \mathscr{E}  \right) \ast \frac{t^{l}}{\Gamma(l+1)} \right)(t_n)  \right\|
\leq
c\tau^{l+1}t_{n}^{-m} + c\tau^{k}t_{n}^{l+1-k-m}
\end{equation}
with $0 \leq l \leq k+m-1$.

Next, we will prove the following inequality \eqref{3.0003} for any $t>0$
\begin{equation}\label{3.0003}
\left\|\left( \left( {_{b}\mathscr{E}_{\tau}} + {_{s}\mathscr{E}_{\tau}}  -  \mathscr{E}  \right) \ast \frac{t^{k+m-1}}{\Gamma(k+m)} \right)(t)\right\| \leq c\tau^{k} \quad \forall t\in (t_{n-1},t_{n}).
\end{equation}
Applying Taylor series expansion of $\mathscr{E}(t)$ at $t=t_{n}$, it yields
\begin{equation*}
\begin{split}
\left( \mathscr{E} \ast \frac{t^{k+m-1}}{\Gamma(k+m)} \right)(t)
&= \sum_{l=0}^{k+m-1}   \frac{(t-t_{n})^{l}}{\Gamma(l+1)} \left( \mathscr{E} \ast \frac{t^{k+m-l-1}}{\Gamma(k+m-l)} \right)(t_{n})\\
&\quad +\frac{1}{\Gamma(k+m)} \int^{t}_{t_{n}}(t-s)^{k+m-1} \mathscr{E}(s)ds,
\end{split}
\end{equation*}
which also holds for $\left( {_{b}\mathscr{E}_{\tau}} \ast t^{k+m-1} \right)(t) $ and $\left( {_{s}\mathscr{E}_{\tau}} \ast t^{k+m-1} \right)(t) $.

From \eqref{aadd3.8} and the above equation, for $t\in (t_{n-1},t_n)$, one has
\begin{equation*}
\left\|(t-t_{n})^{l} \left(\left( {_{b}\mathscr{E}_{\tau}} + {_{s}\mathscr{E}_{\tau}} -  \mathscr{E}  \right) \ast t^{k+m-l-1} \right)(t_n)  \right\|  \leq c\tau^l \left( \tau^{k+m-l}t_{n}^{-m} + \tau^{k}t_{n}^{-l}  \right)
 \leq  c \tau^{k}.
\end{equation*}
Using  \eqref{nas3.007}, \eqref{resolvent estimate}, and \eqref{eq3.3}, it leads to
\begin{equation*}
\| \mathscr{E}(t) \| \leq c \left( \int^{\infty}_{\kappa}e^{rt\cos\theta} r^{m-1} dr + \int^{\theta}_{-\theta}e^{\kappa t \cos\psi}\kappa^{m} d\psi \right) \leq c t^{-m},
\end{equation*}
and
\begin{equation*}
\left\| \int^{t}_{t_{n}}(t-s)^{k+m-1} \mathscr{E}(s)ds \right\| \leq c \int^{t_{n}}_{t}(s-t)^{k+m-1} s^{-m}ds \leq  c \tau^{k}.
\end{equation*}

On the other hand, using the estimates of $ {_{b}\mathscr{E}_{\tau}}(t)=\sum^{\infty}_{n=0}{_{b}\mathscr{E}^{n}_{\tau}}\delta_{t_{n}}(t)$ and $ {_{s}\mathscr{E}_{\tau}}(t)=\sum^{\infty}_{n=0} {_{s}\mathscr{E}^{n}_{\tau}} \delta_{t_{n}}(t+t_{m})$ in \eqref{3.0002} and \eqref{eq3.9}, respectively, for any $t \in   (t_{n-1},t_{n})$, we deduce
\begin{equation*}
\left\| \int^{t}_{t_{n}}(t-s)^{k+m-1} {_{b}\mathscr{E}_{\tau}}(s)ds \right\|
 \leq  (t_n-t)^{k+m-1}  \| {_{b}\mathscr{E}^{n}_{\tau}} \|  \leq c \tau^{k+m} t_{n}^{-m}
 \leq c \tau^{k},
\end{equation*}
and
\begin{equation*}
\left\| \int^{t}_{t_{n}}(t-s)^{k+m-1} {_{s}\mathscr{E}_{\tau}}(s)ds \right\|
 \leq  (t_n-t)^{k+m-1}  \left\| {_{s}\mathscr{E}^{n}_{\tau}} \right\|  \leq c \tau^{k+m} t_{n}^{-m}
 \leq c \tau^{k}.
\end{equation*}
Then we obtain \eqref{3.0003}.
From \eqref{nas3.6},  \eqref{nas3.8} and \eqref{3.0003}, we have
\begin{equation*}
\begin{split}
\left\|V(t_{n})-V^{n}\right\|_{L^2(\Omega)}
& = \left\| \left(\left( \left( {_{b}\mathscr{E}_{\tau}} + {_{s}\mathscr{E}_{\tau}}  -  \mathscr{E}  \right) \ast \frac{t^{k+m-1}}{\Gamma(k+m)} \right) \ast \partial_{t}^{k} f(t)  \right) (t_{n})\right\|_{L^2(\Omega)} \\
& \leq c\tau^{k} \int_{0}^{t_{n}}  \left\|\partial_{s}^{k} f(s) \right\|_{L^2(\Omega)}ds.
\end{split}
\end{equation*}
The proof is completed.
\end{proof}

\begin{theorem}\label{Theorem3.3}
Let $V(t_{n})$ and $V^{n}$ be the solutions of \eqref{RPQ} and \eqref{IDWSBDF}, respectively.
Let $v\in L^{2}(\Omega)$, $f\in C^{k-1}([0,T]; L^{2}(\Omega))$, and $\int_{0}^{t} \left\| \partial_{t}^{k} f(s) \right\|_{L^2(\Omega)} ds <\infty$.
Then the following error estimate holds for any $t_n>0$:
\begin{equation*}
\begin{split}
\left\|V^{n}-V(t_{n})\right\|_{L^2(\Omega)}
& \leq \left( c \tau^{m+1} t_{n}^{-m-1} + c\tau^{k} t_{n}^{-k} \right) \left\| v \right\|_{L^2(\Omega)} \\ 
& \quad + \sum_{l=0}^{k-1}  \left( c\tau^{l+m+1}t_{n}^{-m} + c\tau^{k}t_{n}^{l+1-k} \right) \left\| \partial_{t}^{l} f(0) \right\|_{L^2(\Omega)}  \\
& \quad + c\tau^{k}  \int_{0}^{t_{n}}   \left\| \partial_{t}^{k} f(s) \right\|_{L^2(\Omega)}  ds
\end{split}
\end{equation*}
{ with $1 \leq  m \leq k \leq  7$}.
\end{theorem}
\begin{proof}
According to Theorem \ref{Theorem3.2} and Lemmas \ref{lemma3.9}, \ref{lemma3.10}, the desired result is obtained.
\end{proof}

{
\begin{remark}
From Theorem \ref{Theorem3.3}, we know that the estimate is nothing but $O(\tau)$ if $f$ is merely continuous. On the other extreme, the error is $O(\tau^7)$ when $f$ is of class $C^6$ and $\partial_{t}^{l} f(0) = 0$ for $0 \leq l \leq 6$, with $v=0$. This demonstrates the (parabolic) optimality of the result in Theorem \ref{Theorem3.3}.
Note that $m=k-1$ ($1\leq m \leq k\leq 7$) is a recommended choice to maintain optimal accuracy depending on the smoothness of $f$.
\end{remark}

}

\begin{remark}
	From Theorem \ref{Theorem3.3}, for the model $(*)$,  there exists the corresponding error estimates, namely
	\begin{equation*}
		\begin{split}
			\left\|V^{n}-V(t_{n})\right\|_{L^2(\Omega)}
			& \leq \left( c \tau^{m+1} t_{n}^{-m-1} + c\tau^{k} t_{n}^{-k} \right) \left\| v \right\|_{L^2(\Omega)} \\ 
			& \quad +  \left( c\tau^{m+1}t_{n}^{-m} + c\tau^{k}t_{n}^{1-k} \right) \left\| \chi_{_D} \right\|_{L^2(\Omega)}.
		\end{split}
	\end{equation*}
\end{remark}

\section{Numerical experiments}\label{Sec6}

We discretize the spatial direction by the spectral collocation via the Chebyshev--Gauss--Lobatto points \cite{Shen2011Spectral}.
The discrete $L^2$-norm ($||\cdot||_{l_2}$) is utilized to measure the numerical errors at the terminal time, e.g., $t=t_N=1$.
Since the analytic solution is unknown, the convergence rate of the numerical results is computed by
\begin{equation*}
{\rm Convergence ~Rate}=\frac{\ln \left(||u^{N/2}-u^{N}||_{l_2}/||u^{N}-u^{2N}||_{l_2}\right)}{\ln 2}.
\end{equation*}
All numerical experiments are programmed using Julia 1.8.0.
{ Note that in the numerical scheme \eqref{IDWSBDF1}, the factor $\tau^m$ in the denominator can lead to significant round-off errors when $m$ is large.
In our computations, we use multiple-precision floating-point arithmetic to mitigate round-off errors during evaluation.}

For the sake of brevity, we primarily  employ the corrected ID$m$-WSBDF$k$  with $\beta=3$, $k=7$  in \eqref{IDWSBDF} for simulating the model \eqref{PQ}, since the similar numerical results can be obtained for  $1\leq m\leq k<7$.
Let $T=1$ and $\Omega=(-1, 1)$.
We consider the following two examples to illustrate the convergence analysis.
\begin{description}
  \item[(a)] $v(x)= \sqrt{1-x^2}$ and  $f = 0$.
  \item[(b)] $v(x)= \sqrt{1-x^2}$ and $f(x,t)= \cos(t) \left( 1 +  \chi_{\left(0,1\right)}\left(x\right) \right)$.
\end{description}

\begin{table}[H]  
{\small \begin{center}
\caption{The discrete $L^2$-norm $||u^{N}-u^{2N}||$ and convergence order for case \textbf{(a)}.}
\begin{tabular}{|c| c c c c c c|}
\hline
  $m$    & $N=200$     & $N=400$     &  $N=800$    &  $N=1600$   &  $N=3200$   & Rate     \\ \hline
  0      & 2.0460e-03  & 9.9978e-04  & 4.9424e-04  & 2.4572e-04  & 1.2251e-04  & $\approx$ 1.00     \\ \hline
  1      & 1.2056e-06  & 3.0138e-07  & 7.5345e-08  & 1.8836e-08  & 4.7090e-09  & $\approx$ 2.00     \\ \hline
  2      & 2.2042e-07  & 2.6721e-08  & 3.2960e-09  & 4.0927e-10  & 5.0990e-11  & $\approx$ 3.00     \\ \hline
  3      & 1.3827e-08  & 5.3242e-10  & 3.2759e-11  & 2.0315e-12  & 1.2647e-13  & $\approx$ 4.00     \\ \hline
  4      & 8.6202e-09  & 9.2812e-12  & 2.8477e-13  & 8.8183e-15  & 2.7431e-16  & $\approx$ 5.00     \\ \hline
  5      & 1.8092e-07  & 1.5535e-13  & 2.3739e-15  & 3.6703e-17  & 5.7048e-19  & $\approx$ 6.00     \\ \hline
  6      & 1.0977e-06  & 1.0638e-14  & 1.9259e-17  & 1.4865e-19  & 1.1544e-21  & $\approx$ 7.00     \\ \hline
  7      & 5.1616e-06  & 5.1280e-14  & 4.9954e-19  & 3.2626e-21  & 2.3038e-23  & $\approx$ 7.14     \\ \hline
\end{tabular}\label{table:6.1}
\end{center}}
\end{table}

\begin{table}[H]  
{\small \begin{center}
\caption{The discrete $L^2$-norm $||u^{N}-u^{2N}||$ and convergence order for case \textbf{(b)}.}
\begin{tabular}{|c| c c c c c c|}
\hline
  $m$    & $N=200$     & $N=400$     &  $N=800$    &  $N=1600$   &  $N=3200$   & Rate      \\ \hline
  0      & 6.4833e-04  & 3.1681e-04  & 1.5661e-04  & 7.7866e-05  & 3.8823e-05  & $\approx$ 1.00     \\ \hline
  1      & 3.8207e-07  & 9.5502e-08  & 2.3875e-08  & 5.9688e-09  & 1.4922e-09  & $\approx$ 2.00     \\ \hline
  2      & 7.0592e-08  & 8.4679e-09  & 1.0444e-09  & 1.2969e-10  & 1.6158e-11  & $\approx$ 3.00     \\ \hline
  3      & 1.0420e-08  & 1.6885e-10  & 1.0389e-11  & 6.4423e-13  & 4.0107e-14  & $\approx$ 4.00     \\ \hline
  4      & 8.5170e-09  & 2.9801e-12  & 9.1317e-14  & 2.8259e-15  & 8.7884e-17  & $\approx$ 5.00     \\ \hline
  5      & 1.6041e-07  & 5.8791e-14  & 8.8085e-16  & 1.3505e-17  & 2.0908e-19  & $\approx$ 6.01     \\ \hline
  6      & 9.7004e-07  & 9.5145e-15  & 1.5120e-17  & 1.1326e-19  & 8.6659e-22  & $\approx$ 7.03     \\ \hline
  7      & 1.4614e-04  & 3.0090e-12  & 4.7157e-16  & 3.5026e-18  & 2.6685e-20  & $\approx$ 7.03     \\ \hline
\end{tabular}\label{table:6.3}
\end{center}}
\end{table}

Tables \ref{table:6.1} and \ref{table:6.3} demonstrate that the corrected ID$m$-WSBDF$7$ method achieves high-order convergence with $\mathcal{O}(\tau^{\min\{m+1,7\}})$, which is in  agreement with Theorems \ref{Theorem3.2} and \ref{Theorem3.3}, respectively. Furthermore, applying Theorem \ref{Theorem3.2} again implies  that the WSBDF$k$ method in \eqref{WSBDF} exhibits only first-order convergence, since  $m=0$ (see Table \ref{table:2.1}). It's worth noting that the ID$1$-WSBDF$7$ method \eqref{IDWSBDFad1} restores second-order convergence, as observed in Table \ref{table:2.2}, owing to its equivalence with the corrected ID$1$-WSBDF$7$ method.




\end{document}